\documentclass[11pt]{article}
\usepackage{amsmath}
\usepackage{graphicx}
\usepackage{enumerate}
\usepackage{natbib}
\usepackage{url} 

\usepackage{tipa}
\usepackage{kashsty}
\usepackage{kashthm}

\usepackage{algorithm}
\usepackage{algorithmic}

\usepackage[margin=1.37in]{geometry}

\def\PICDIR{.}

\begin{document}



  \title{\bf Functional Response Designs via the Analytic Permutation Test }
  \author{Adam B Kashlak and Sergii Myroshnychenko\thanks{
    The authors gratefully acknowledge 
    the Natural Sciences and Engineering Research Council of 
    Canada and the Pacific Institute for the Mathematical Sciences
    }\hspace{.2cm}\\
    Department of Mathematical and Statistical Sciences,\\ 
    University of Alberta\\
    and \\
    Susanna Spektor \\
    Pilon School of Business,\\The Sheridan College Institute of Technology and Advanced Learning}
  \maketitle

\begin{abstract}
Vast literature on experimental design extends from Fisher 
and Snedecor to the modern day.  When data lies beyond the 
assumption of univariate normality, nonparametric methods 
including rank based statistics and permutation tests are 
enlisted.
The permutation test is a versatile exact nonparametric
significance test that requires drastically fewer assumptions
than similar parametric tests.  
The main downfall of the permutation test is 
high computational cost 
making this approach laborious for complex data and 
sophisticated experimental designs
and completely infeasible in any application requiring speedy
results such as high throughput streaming data.
We rectify this problem through application of concentration 
inequalities and thus 
propose a computation free permutation test---i.e. a 
permutation-less permutation test.
This general framework is applied to multivariate, 
matrix-valued, and functional data.  
We improve these concentration bounds
via a novel incomplete beta transform.
We extend our theory from 2-sample to $k$-sample testing
through the use of weakly dependent Rademacher chaoses and 
modified decoupling inequalities.
We test this methodology on classic functional data sets including
the Berkeley growth curves and the phoneme dataset. 
We further consider analysis of spoken vowel sound under 
two experimental designs: 
the Latin square and the randomized block design.
\end{abstract}

\noindent%
\vfill

\section{Introduction}

Exact significance tests date back to the very origins of 
statistical hypothesis testing as an alternative to parametric
testing.  Namely, Fisher's exact test
tests for independence between the rows or columns of a 
$2\times2$ contingency table by directly using the hypergeometric
distribution instead of relying on large sample asymptotic
statistics such as the chi-squared test.  
As a consequence, it obtains the 
exact p-value of the data without relying on large sample
asymptotics.  However, Fisher's exact test is severely limited
as extension to general $r\times c$ tables requires significant
amounts of computational power to enumerate or approximate the 
entire discrete distribution 
\citep{GOOD1956,AGRESTI1992}.

Permutation tests comprise a large subclass of such exact significance
tests and have been thoroughly studied
\citep{MIELKE2007,BASSO2009,PESARIN2010,BROMBIN2013,GOOD2013}.
Given a sample 
$X = \left\{X_1,\ldots,X_n\right\}\in\mathcal{X}$ 
for some measure space $\mathcal{X}$, 
a permutation test considers the finite sampling 
distribution 
of a test statistic $T(X)$ over a discrete group where
the distribution of $T$ is invariant 
for any group action on the observed data
\citep{KALLENBERG2006}---i.e.
for a group $G$, 
$T(gX)\eqdist T(X)$ for any $g\in G$.
A canonical example is one-way ANOVA;
see \cite{BASSO2009} Section 5.2 for more details.  

The permutation test requires far fewer assumptions than standard
parametric approaches---namely that of exchangability under the 
null hypothesis---and 
is thus robust against deviations from distributional 
assumptions like normality and
provides guaranteed performance for finite samples.  The main
limitation is that of computation.  Performing a two sample
permutation test for real valued data is trivial with modern
computers.  What if we were to perform a $k$ sample test
with $k\choose2$ post-hoc comparisons taking multiple testing
into account for, say, covariance operators as in \cite{PIGOLI2014,CABASSI2017}
where every permutation requires computation of the singular value 
decomposition (SVD) of a large matrix?
Furthermore, what if we desire a more sophisticated experimental 
design such as a randomized block, Latin square, or unreplicated 
factorial design with the addition of multiple testing corrections?
The amount of computation required to get accurate p-values will 
be prohibitive.  The speech data and design considered in Section~\ref{sec:vowelData}
would, for example, require 264 SVDs per permutation and with 66 
hypotheses to test at, say, 2000 permutations each requires
nearly 35 million SVDs.  For matrices with dimension $100\times100$, 
this would take an estimated 36 hours on a 
Intel Core i7-7567U CPU at 3.50GHz.  For a $400\times400$ matrix,
it would take 74 days.

In this article, we present a unified methodology for performing
computation-free permutation tests for $k$ sample testing
in commutative and non-commutative $L^q$ spaces, which 
includes multivariate and functional data and 
covariance matrices and operators among other data types.  
Specifically, we consider the distribution of a test
statistic on a discrete space of invariant group actions.  Instead 
of taking random draws from that space to get a
conditional Monte Carlo 
estimate \citep{HEMERIK2018} 
of the p-value, we apply
recent extensions of the Kahane-Khintchine inequality for commutative
and non-commutative Banach spaces 
\citep{PISIERXU2003,GARLING2007,SPEKTOR2016}
in order to
achieve sub-Gaussian bounds on the tail probability of our test 
statistic.  Namely, we seek a result like
$
\prob{ T(\pi) \ge T^\star } \le \exp( -Ct^2 )
$
for some universal constant $C>0$ depending only on the space
in which the data lives irrespective of sample size and dimension.
This methodology is presented in 
Section~\ref{sec:twoSamp} for two sample testing 
within commutative $L^q$ spaces---e.g. 
univariate, vector valued, and 
functional data---as well as within 
non-commutative 
$L^q$ spaces---e.g.
covariance matrices and operators.
As such universal constants are often less than optimal for 
statistical use, we introduce an adjustment
for these upper bounds based on Talagrand's concentration
inequality \citep{TALAGRAND1996CON} and the incomplete beta
function in Section~\ref{sec:betaAdj}. 
An extension to testing on $k$-samples is considered
in Section~\ref{sec:kSamp} making use of Rademacher chaoses and 
decoupling inequalities \citep{KWAPIEN1987,DELAPENAGINE2012}.
Section~\ref{sec:vowelData} extends these ideas to multi-factor
designs for the analysis of spoken vowel phonemes.

Most previous work on fast or computation-free permutation testing
focus on univariate data
in the setting of large scale testing typically applied to 
testing for genomics data.
The recent
work of \cite{HE2019} achieves this goal 
by using Stolarsky's invariance principle.
In \cite{YANG2019}, ``very small'' p-values are approximated 
via sequential Monte Carlo and the Edgeworth expansion.
In \cite{SEGAL2018}, an asymptotic approximation and 
a clever partitioning/resampling scheme is used to 
approximate small p-values.
Density approximation via Pearson curves \citep{SOLOMON1978}
has recently reemerged for p-value approximation in 
machine learning \citep{GRETTON2012} 
and neuroimaging \citep{WINKLER2016} among other areas.
While past work is focused on large scale two-sample testing, 
this work is motivated by $k$-sample tests and
more sophisticated experimental designs with functional and
operator responses.
While permutation tests have been used 
both for pointwise
and curve-wise analysis of functional data
\citep{COX2008,CORAIN2014,CHAKRABORTY2015,PIGOLI2014,PIGOLI2018,CABASSI2017}, 
approaching statistical hypothesis testing via 
analytic estimation of a
permutation test p-value in general Banach spaces
has not been deeply explored as of yet.

As a proof of concept for testing within commutative and 
non-commutative $L^q$ spaces, we consider a variety of 
simulated and real data sets in Section~\ref{sec:otherData}.
In Section~\ref{sec:vowelData}, our bounds are applied 
to testing for phonological differences
among twelve spoken vowel sounds performed as a complete randomized
block design on covariance operators with respect to two 
binary blocking factors:
the speaker's country of origin $\{\text{Canada}, \text{China}\}$ and sex
$\{\text{male}, \text{female}\}$.  We also consider a Latin 
square design for checking the data for within subject 
pronunciation changes while running the experiment.  
Section~\ref{sec:vowelData} contains more
detail on the data, experimental design, and its results.
Proofs of the main theorems, the necessary theoretical development, 
further data experiments,
and a discussion of past results are contained in the 
supplementary material.


\section{Two sample testing}
\label{sec:twoSamp}

\subsection{Univariate data}
\label{sec:univariate}

Let $n=m_1+m_2$ and $X_1,\ldots,X_n\in\real$ be independent random
variables such that $\xv X_i = \mu_1$ for $i\le m_1$ and
$\xv X_i = \mu_2$ for $i\ge m_1+1$.
We wish to test $H_0:\mu_1=\mu_2$ versus $H_1: \mu_1\ne\mu_2$.
To test these hypotheses using a permutation test,
we treat $X_1,\ldots,X_n\in\real$ as fixed and consider
$\pi\in\mathbb{S}_n$ a random permutation uniformly distributed
on the symmetric group on $n$ elements.  That is,
$\pi$ is a bijective map
$\pi:\{1,\ldots,n\}\rightarrow\{1,\ldots,n\}$.
Thus, we can consider the randomly permuted test statistic
\begin{equation}
\label{eqn:tstatPerm}
T(\pi) = 
\frac{1}{s_n}
\left[\frac{1}{m_1}\sum_{i=1}^{m_1} X_{\pi(i)} -
\frac{1}{m_2}\sum_{i=m_1+1}^{n} X_{\pi(i)}\right],
\end{equation}
which is normalized by the sample standard deviation $s_n$ for
the entire set $X_1,\ldots,X_n$.\footnote{Note that $s_n$ is
	invariant under permutation and is only included to make the 
	below formulation nicer.}
The conditional tail probability is
\begin{equation}
\label{eqn:permPVttest}
\prob{T(\pi)\ge t \,|\, X_1,\ldots,X_n} = 
\frac{1}{n!}\sum_{\pi\in\mathbb{S}_n} \boldsymbol{1}[{T(\pi)\ge t}].
\end{equation}
Let $T_0$ be the test statistic $T(\pi)$ when 
$\pi$ is the identity---i.e. the original ordering.  
Then, the p-value for the above hypothesis
test is $\prob{T(\pi)\ge T_0}$, which is often approximated by randomly
generating $N \ll n!$ random permutations from $\mathbb{S}_n$ 
instead of exhaustively enumerating all elements of $\mathbb{S}_n$.
This results in an overly conservative test for p-values 
approaching $1/N$.

To avoid the simulation-based approximation of 
equation~\ref{eqn:permPVttest},
we instead prove a sub-Gaussian bound on the p-value.

\renewcommand{\theenumi}{\bf \alph{enumi}}
\begin{theorem}[Univariate Data]
	\label{thm:uniReal}
	For $T(\pi)$ from equation~\ref{eqn:tstatPerm}
	with $m_1 = \kappa m_2$ for some $\kappa\ge1$, then
	$
	\prob{T(\pi)\ge t} \le \exp\left(  
	-nt^2/2\lceil\kappa+1\rceil^3
	\right).
	$
\end{theorem}


This theorem is extended to more advanced settings, including 
vectors, matrices, functional data, and operators, in the 
following sections.  We state those subsequent 
theorems for balanced
samples, but note that the imbalanced setting of 
Theorem~\ref{thm:uniReal} can also be incorporated 
with a similar constant $\kappa$.  The proofs are more
tedious for imbalanced data, but no additional innovation 
is required.

\subsection{Commutative $L^q$ Spaces}

To extend our tail bounds beyond the real valued setting, we require
some definitions.
Note that both of the following definitions extend to the case 
of compact operators on Banach spaces.

\begin{definition}[Matrix Square Root]
	\label{def:sqrtmat}
	Let $A\in\real^{d\times d}$ with $d\ge2$ 
	be a symmetric positive semi-definite matrix
	with eigen-decomposition $A = UD\TT{U}$ 
	where $U=(v_1~v_2~\ldots~v_d)$ 
	is the orthonormal matrix of eigenvectors and $D$ is the diagonal
	matrix of eigenvalues, $(\lmb_1,\ldots,\lmb_d)$.  
	Then, $A^{1/2} = UD^{1/2}\TT{U}$
	where $D^{1/2}$ is the diagonal matrix with entries
	$(\lmb_1^{1/2},\ldots,\lmb_d^{1/2})$.
\end{definition}
\begin{definition}[$q$-Schatten norm for matrices]
	For an arbitrary matrix $A\in\real^{k\times l}$ and $q\in(1,\infty)$, 
	the $q$-Schatten norm is
	$
	\norm*{A}_{S^q}^q = \mathrm{tr}[(\TT{A}A)^{q/2}]
	= \norm{{\boldsymbol \nu}}_{\ell^q}^q
	= \sum_{i=1}^{\min\{k,l\}} \nu_i^q
	$
	where ${\boldsymbol \nu}=(\nu_1,\ldots,\nu_{\min\{k,l\}})$
	is the vector of singular values of $A$ and where
	$\norm{\cdot}_{\ell^q}$ is the standard $\ell^q$ norm 
	in $\real^d$.
	In the covariance matrix case where $A\in\real^{d\times d}$
	is symmetric and positive-definite,
	$
	\norm*{A}_{S^q}^q = \tr{A^q}
	= \norm{{\boldsymbol \lmb}}_{\ell^q}^q
	$
	where ${\boldsymbol \lmb}$ is the vector of eigenvalues of $A$.
	
	When $q=\infty$, we have the standard operator norm on 
	$\ell^2(\real^d)$,  
	$
	\norm{A}_{S^\infty} = 
	\sup_{v\in\real^d,\norm{v}_{\ell^2}=1} \norm{A v}_{\ell^2} = 
	\sup_{v\in\real^d,\norm{v}_{\ell^2}=1} \TT{v}A v.
	$
	In the covariance matrix setting, this coincides with the 
	maximal eigenvalue of $A$.
\end{definition}

Let $X_1,\ldots,X_n\in\mathcal{X}$ where 
$\{\mathcal{X},\norm{\cdot}\}$ is a commutative $L^q$ space.
The test statistic of interest
is $T_0 = \norm{\sum_{i\le m}X_i - \sum_{i>m}X_i}$.
Then, 
Theorem~\ref{thm:uniReal} can be extended to such settings
using a version of the Kahane-Khintchine inequality under
a weak dependency condition from Theorem~A.7 
proved in the supplementary material.  For simplicity of notation, 
we assume that the $X_i$ are centred about the sample mean
and that the samples are balanced.

\begin{theorem}[Commutative $L^q$ Spaces]
	\label{thm:uniComm}
	Let $m = n/2$,
	$\norm{\cdot}_{S^q}$ be the $q$-Schatten norm for
	matrices or operators,
	and $\veps_1,\ldots,\veps_n$ be Rademacher random variables
	such that $\sum_{i=1}^n\veps_i=0$.  
	Let $X_1(t),\ldots,X_n(t)$ be continuous function on a
	compact interval
	with empirical covariance operator
	$\hat{\Sigma}(s,t) = (n-1)^{-1}\sum_{i=1}^nX_i(s)X_i(t)$.
	Let $q\in[1,\infty)$
	with norm $\norm{\cdot}_{L^q}$.  For
	$T(\pi) = \norm*{\sum_{i=1}^n \veps_i X_i}_{L^q}$.
	Then,
	$
	\prob{T(\pi)\ge t} \le \exp\left(  
	-t^2/c\norm{ \hat{\Sigma}^{1/2} }_{S^q}^2
	\right).
	$
\end{theorem}

\begin{remark}[On optimal constants]
	The optimal constant $c$ in the above theorem follows from
	the optimal constant in the Kahane-Khintchine inequality, which 
	is not currently known.\footnote{It took about 60 years from 
		the advent of the original Khintchine inequality for optimal
		constants to be determined.}
	However, it is strongly conjectured to agree with the 
	optimal constant for the standard Khintchine inequality.  
	In that case, we
	would take $c=64$ in the above theorem, which is $16$ from
	Theorem~\ref{thm:uniReal} times $4$ from that fact that 
	$T(\pi)$ is not a symmetric random variable.  For more 
	details, see the proof and discussion in 
	the supplementary material.
	We also empirically adjust the p-values in 
	Section~\ref{sec:betaAdj}, which is demonstrated to
	give strong performance in Sections~\ref{sec:otherData}
	and~\ref{sec:vowelData}.
\end{remark}

\subsection{Non-Commutative $L^q$ Spaces}

Following from the previous section, we outline similar 
tail bounds in non-commutative $L^q$ spaces 
\citep{PISIERXU2003}.
This methodology encompasses matrix and operator data with
emphasis on application to testing for equality of 
covariances.  Hence, the following theorem is applied to
symmetric positive definite operators in the example below
and to the data in Section~\ref{sec:vowelData}.
The test statistic of interest
is still $T_0 = \norm{\sum_{i\le m}X_i - \sum_{i>m}X_i}$, 
but with the $X_i$ now belonging to a non-commutative $L^q$ space.

\begin{theorem}[Non-Commutative $L^q$ Spaces]
	\label{thm:uniNonComm}
	Let 
	$\norm{\cdot}_{S^q}$ be the $q$-Schatten norm for a
	matrix or operator and $\veps_1,\ldots,\veps_n$ be 
	Rademacher random variables
	such that $\sum_{i=1}^n\veps_i=0$. 
	For $d,d'>1$, let $X_1,\ldots,X_n\in\real^{d\times d'}$
	be a collection of $n$ fixed matrices (or 
	let $X_1,\ldots,X_n$ be a collection of bounded
	linear operators).
	For
	$T(\pi) = \norm*{\sum_{i=1}^n \veps_i X_i}_{S^q}$,
	there exists a universal constant $c>0$ such that
	$
	\prob{T(\pi)>t} \le \exp\left(  
	-t^2/c \mathcal{S}^2
	\right)
	$
	where 
	$ 
	\mathcal{S}= \max\left\{
	\norm{ (\sum_{i=1}^nX_i{X}_i^*)^{1/2} }_{S^q},
	\norm{ (\sum_{i=1}^n{X}_i^*{X_i})^{1/2} }_{S^q}
	\right\}
	$
	with $X_i^*$ the adjoint operator.
\end{theorem}

\begin{remark}
	Of particular interest are covariance operators being
	compact trace-class self-adjoint operators.  Consequently, 
	we have the same bound but with 
	$ 
	\mathcal{S}= 
	\norm{ (\sum_{i=1}^nX_i^2)^{1/2} }_{S^q}.
	$
\end{remark}

\subsection{Beta and Empirical Beta Adjustment}
\label{sec:betaAdj}

Inequalities such as the Kahane-Khintchine inequalities
are useful tools for considering the finite sample
performance of a statistical method.
However, the biggest impediment to the use of such 
inequalities, as well as other concentration inequalities,
for statistical inference is the 
nearly inevitable loss in power to reject the null
due to `universal constants' that are too large for
application.
We thus propose a transformation based on the beta
distribution to correct the p-values and recover
the lost statistical power.  
Proposition~\ref{prop:betaAdj} only applies to univariate
data and requires an asymptotic arguement outside of
our finite sample focus.  It is included nevertheless
to set the stage for the non-asymptotic beta transform
in Theorem~\ref{thm:betaAdj}.  Furthermore, this simpler 
setting yields
explicit beta parameters and demonstrates stellar performance
in both the simulated data of Section~\ref{sec:otherDataUni}
and the extremely imbalanced small sample setting explored in 
\cite{KASHLAK_YUAN_ABELECT} where, for example, 
$m_1=335$ and $m_2=3$.  

For a statistical 
test, if the correct test size is achieved, then 
a random null p-value will be distributed as 
$\distUnifInt{0}{1}$.  However, our Kahane-Khintchine based
null p-values will instead closely follow a
more general $\distBeta{\alpha}{\beta}$ distribution.
Thus, identification of the parameters $\alpha$ and $\beta$
will allow us to adjust the p-values to the null setting to 
recover lost statistical power.
This idea is spiritually similar to the Pearson curve method
\citep{SOLOMON1978}, but that approach requires estimation of the 
first 4 central moments for comparison with the family of generalized
Pearson distributions compared to our more focused use of the 
beta distribution with Theorem~\ref{thm:betaAdj} proved to justify
such focus.  Usage of the Edgeworth expansion \citep{HALL2013}
is another method with a long history, but requires some care
to note whether or not a finite number of terms in such an expansion
can provide a satisfactory approximation to the probability density
in question \citep{KENDALL}.

We first consider the univariate case of 
Section~\ref{sec:univariate} before discussing the more 
general Banach space setting for the beta transform.

\begin{proposition}
	\label{prop:betaAdj}
	Under the setting of Theorem~\ref{thm:uniReal}
	with $n$ sufficiently large,
	$$
	\prob{
		\exp\left\{
		-nT(\pi)^2/2\lceil\kappa+1\rceil^3
		\right\} < u
	} \le C_0I\left(
	u ; 
	\frac{\lceil\kappa+1\rceil^3}{(2+\kappa+\kappa^{-1})},
	\frac{1}{2}
	\right)
	$$
	where $I(u;\alpha,\beta)$ is the regularized incomplete
	beta function and\\  
	$
	C_0 = {\left(\frac{
			\lceil\kappa+1\rceil^{3}
		}{
			2+\kappa+\kappa^{-1}
		}\right)^{1/2}
		\Gamma\left(
		\frac{\lceil\kappa+1\rceil^3}{2+\kappa+\kappa^{-1}}
		\right)}{
		\Gamma\left(
		\frac{1}{2}+\frac{\lceil\kappa+1\rceil^3}{2+\kappa+\kappa^{-1}}
		\right)^{-1}
	}.
	$
\end{proposition}

Proposition~\ref{prop:betaAdj} allows us to adjust the p-values
from Theorem~\ref{thm:uniReal} so that our test statistic 
achieves the desired empirical size.  The refined bound
is 
$$
\prob{T(\pi)>t} \le
C_0     
I\left(
\ee^{  
	-nt^2/2\lceil\kappa+1\rceil^3
};
\frac{\lceil\kappa+1\rceil^3}{2+\kappa+\kappa^{-1}},\frac{1}{2}
\right)
$$
This adjustment is shown to work in the simulations
detailed in Figure~\ref{fig:uniTest}.
For the more general Banach space setting, we can 
use Talagrand's concentration inequality \citep{TALAGRAND1996CON}
to prove the following theorem.

\begin{theorem}
	\label{thm:betaAdj}
	Let $(\mathcal{X},\norm{\cdot})$ be a Banach
	space with separable dual space $\mathcal{X}^*$, 
	and let
	$h:\real\rightarrow\real$ be monotonically
	increasing. 
	For any random 
	variable $X$ taking values in $\mathcal{X}$ such 
	that $\xv h(\norm{X})^2<\infty$ and 
	$\sup_{X\in\mathcal{X}}h(\norm{X})<U<\infty$ and 
	for $u\in(0,1)$ and 
	some constants 
	$C,c,\alpha,\beta>0$,
	$
	\prob{ \ee^{-h(\norm{X})/c} < u } \le C I(u;\alpha,\beta)
	$
	where $I(u;\alpha,\beta)$ is the incomplete beta function
	for $c$ sufficiently large. 
\end{theorem}

\begin{remark}
	Theorem~\ref{thm:betaAdj} requires the Banach space
	$\mathcal{X}$ to have a separable dual.  This stems from
	writing the norm as a countable supremum for use within
	Talagrand's concentration inequality \citep{TALAGRAND1996CON}.
	We can directly apply this result to 
	commutative and non-commutative $L^q$ spaces for 
	$1<q<\infty$.  However, $L^\infty$ is a standard example of
	a non-separable Banach space.  For our purposes, we can avoid 
	this issue as it is typical in functional data analysis 
	to consider the uniform norm on the space of continuous 
	bounded functions with compact support.
\end{remark}

When working in commutative and non-commutative $L^q$
spaces, we no longer have easily defined constants for the 
righthand bound in Theorem~\ref{thm:betaAdj}.
Hence, we instead propose an \textit{empirical beta transform}, 
outlined in Algorithm~\ref{algo:empBetaAlg}, which estimates
these constants.
To do this, we must choose a small number $r$ of permutations to 
draw uniformly at random from $\mathbb{S}_n$.
In practice, we
	find that $r = 10$ or 20 is sufficient to achieve good results 
	on real 
	data.
From these, we compute test statistics sampled from the null
setting, which will yield a collection of $r$ p-values.
These p-values can in turn be used to estimate the parameters
for a beta distribution via the method of moments estimate
$\hat{\alpha}$ and $\hat{\beta}$.  Lastly, the p-value $p_0$ 
produced by $T_0$ can be adjusted by application of the 
incomplete beta function: $I(p_0;\hat{\alpha},\hat{\beta})$.
This method was applied to most of the data examples detailed in
Section~\ref{sec:otherData}.  
This transform is also shown to work well for the construction
of wild bootstrap confidence regions for least squares and 
ridge regression \citep{KASHLAK_BURAK_WILDBS}.

\begin{algorithm}[t]
	\caption{
		\label{algo:empBetaAlg}
		The Empirical Beta Transform
	}
	\begin{tabbing}
		\qquad \enspace Compute p-value $p_0$ from test
		statistic $T_0$ using Theorem~\ref{thm:uniComm}
		or~\ref{thm:uniNonComm}.\\
		\qquad \enspace Choose $r>1$, the number of permutations 
		to simulate---e.g. $r=10$.\\
		\qquad \enspace Draw $\pi_1,\ldots,\pi_r$ from $\mathbb{S}_n$
		uniformly at random.\\
		\qquad \enspace  Compute p-values 
		$p_1,\ldots,p_r$ from test statistics 
		$T_{\pi_1},\ldots,T_{\pi_r}$.\\ 
		\qquad \enspace Find the method of moments estimator
		for $\alpha$ and $\beta$.  \\
		\qquad\qquad Estimate first and second central moments of the 
		$p_i$ by $\bar{p}$ and $s^2$.\\ 
		\qquad\qquad Estimate 
		$\hat{\alpha} = {\bar{p}^2(1-\bar{p})}/{s^2}-\bar{p}$.\\
		\qquad\qquad Estimate
		$\hat{\beta} = [ \bar{p}(1-\bar{p})/s^2 - 1 ][ 1-\bar{p} ]$.\\
		\qquad \enspace Return the adjusted p-value
		$I(p_0; \hat{\alpha},\hat{\beta})$.
		
	\end{tabbing}
\end{algorithm}

\section{k sample testing}
\label{sec:kSamp}

For general one-way ANOVA and more complex experimental
designs, we extend the above two sample tests to $k$ level 
factors.  The two challenges to overcome are (1) proper 
multiple testing correction for the $k\choose2$ pairwise
comparisons and (2) the construction of a global p-value.
Classical hypothesis testing would have us first reject the 
global hypothesis and follow up with pairwise post-hoc testing.
For permutation tests, we begin with pairwise testing and 
combine these tests into a global p-value.

For one-way ANOVA,
let $X_{i,j}$ be the $j$th observation from category $i$ 
for $i=1,\ldots,k$ and $j=1,\ldots,n_i$ under the model
\begin{equation}
\label{eqn:onewayANOVA}
X_{i,j} = \mu + \tau_{i} + \veps_{i,j}
\end{equation}
with global mean $\mu$, $i$th treatment effect $\tau_i$ 
with $\sum_{i=1}^k\tau_i=0$, and exchangeable errors 
$\veps_{i,j}$---i.e. permutationally invariant 
\citep{KALLENBERG2006}.
We wish to test the following:
\begin{align*}
&\text{Pairwise}&&H_0^{(ij)}: \tau_i=\tau_j & &H_1^{(i,j)}: \tau_i\ne\tau_j\\
&\text{Global}&&H_0: \tau_1=\ldots=\tau_k=0 & 
&H_1: \exists \tau_i\ne0.
\end{align*}
Under the pairwise null $H_0^{i,j}$, the difference in category means
is $\bar{X}_{i\cdot}-\bar{X}_{j\cdot}=
\bar{\veps}_{i\cdot}-\bar{\veps}_{j\cdot}$.  Thus, the permutation
test requires exchangeable errors---i.e. the distribution 
of 
$\bar{\veps}_{i\cdot}-\bar{\veps}_{j\cdot}$
is invariant under any random permutation.  This is weaker
than the standard iid setup and, most critically, does
not require normality.

\subsection{Multiple Pairwise Tests}

From Section~\ref{sec:twoSamp}, we can compute test statistics 
$T_0^{(ij)}$ for $H_0^{(ij)}$ and consider the permutation
distribution of $T^{(ij)}(\pi)$ for some uniformly
distributed $\pi\in\mathbb{S}_{n_i+n_j}$.  
For familywise type I error control, 
the pairwise statistics come from independent applications
of dependent Rademacher vectors.
Hence, we can rely on standard
multiple testing corrections such as the simple Bonferroni 
correction as proposed in \cite[Chapter 5]{BASSO2009} or the slightly
more involved step-down procedure used in \cite{CABASSI2017}.
Other methods include Holm's stepdown method \citep{HOLM1979}, 
the approach outlined in the 
canonical text \cite{WESTFALL_YOUNG}, and the more 
recent \cite{ROMANO2005}.
In experimental design, some authors even prefer to forego such 
corrections and report raw uncorrected p-values
\citep{WUHAMADA}.  The focus of this article is on 
computation of the raw p-values and hence, application 
of one's favourite multiple testing correction is left
to the reader.
For the phonological data analysis in Section~\ref{sec:vowelData},
we will just consider the raw p-values and the Bonferroni correction. 

\subsection{Global Test}

The k-sample global significance test statistic can be written
as a combination of the pairwise statistics:
\begin{equation}
\label{eqn:globalTest}
T_0 = \sum_{i=1}^{k-1}\sum_{j=i+1}^k n_in_j (T_0^{(ij)})^2
\end{equation}
To test the significance of $T_0$, a permutation framework
can be implemented in one of three ways; see
\cite{BASSO2009} Chapter~5 for more details.  The first is the 
pooled method
in which the entire data set of $N = n_1+\ldots+n_k$ 
observations is permuted.  The second is by aggregation
of the pairwise statistics where each permutation is applied
independently to each pair of samples. 
The third is the 
synchronized method which only applies to balanced
designs---i.e. $n_1=\ldots=n_k$---in which the 
same permutations are applied to each category pairing $(i,j)$.
This is the preferable approach when the design is balanced
\citep{BASSO2009}.
As we have already discussed individual pairwise testing, we
focus on the synchronized test in the context of our 
Kahane--Khintchine methodology.

\begin{remark}
	Beyond univariate data, the above test statistic $T_0$ can 
	be considered on the direct sum of $\kappa = {k\choose2}$ 
	Banach spaces.
	That is, for a sequence of Banach spaces 
	$(\mathcal{X}_i,\norm{\cdot}_i)$ and elements
	$X_i\in\mathcal{X}_i$, we can define a new
	Banach space by the
	$\ell^2$ direct sum 
	$
	(X_i)_{i=1}^n\in
	\left( \bigoplus_{i=1}^{\kappa} \mathcal{X}_i \right)_{\ell^2}
	$
	with norm 
	$\norm{ (X_i)_{i=1}^n } = (\sum_{i=1}^n \norm{X_i}_i^2)^{1/2}$.
	See any text on discussing sequences in Banach spaces
	such as \cite{DIESTEL1995} for more details.
\end{remark}

The synchronized setting is the preferred approach for balanced designs;
see, for example, \cite{BASSO2009,CABASSI2017}.
This approach applies the same permutations to each pairing.
Let $X^{(1)},\ldots,X^{(k)}$ be $m$-long column vectors containing 
the observations of samples $1,\ldots,k$, respectively.  Then, let
$X$ be the $2m \times {k\choose2}$ matrix with columns 
$$
X = \begin{pmatrix}
X^{(1)} & X^{(1)} & \ldots & X^{(k-1)} \\
X^{(2)} & X^{(3)} & \ldots & X^{(k)}
\end{pmatrix}.
$$
Lastly, let $\TT{\veps} = (\veps_1,\ldots,\veps_{2m})$ such that 
$\sum \veps_i=0$.  The synchronized permuted version of the 
global test statistic in Equation~\ref{eqn:globalTest} is then
$
\norm{\TT{X}\veps}_{\ell^2}^2 = \sum_{i,j=1}^{2m} a_{i,j}\veps_i\veps_j
$
for $a_{i,j}$, the $i,j$th entry in $X\TT{X}$.  This is a second
order Rademacher chaos \citep[Section 4.4]{LEDOUXTALAGRAND1991} 
except that the $\veps_i$ are not iid.  In this case, we still
have a sub-Gaussian bound achievable via a decoupling
argument \citep{KWAPIEN1987} with proof in the supplementary material.
See \cite{DELAPENAGINE2012} for more on decoupling inequalities.

\begin{theorem}
	\label{thm:syncTest} 
	Let $T = \norm{\TT{X}\veps}_{\ell^2}$ for $X$ the above 
	$2m\times{k\choose2}$ matrix and $\veps_i$ such that
	$\sum_{i=1}^{2m}\veps_i=0$. Then, for some universal constant $c$,
	$
	\prob{T >t} \le \exp\left[  
	-t^2/c\mathcal{S}
	\right]
	$
	where 
	$\mathcal{S} = \norm{X\TT{X}}_{S^2}$.
\end{theorem}

\begin{remark}
	Up to constant $c$, this theorem coincides with the result for 
	a two sample test as for 
	$X = (x_1^{(1)},\ldots,x_m^{(1)},x_1^{(2)},\ldots,x_m^{(2)})^\text{T}$,
	the term $\mathcal{S}$ equals the sample variance of the $x_{i}^{(j)}$.
	However, the constant $c$ emerging from the proof is very large.
	This universal constant problem is rectified via the empirical
	beta transform presented in Section~\ref{sec:betaAdj}.
\end{remark}

\section{Data Examples}
\label{sec:otherData}

\subsection{Univariate Data}
\label{sec:otherDataUni}

\subsubsection{Two Sample Test}
The performance of Theorem~\ref{thm:uniReal} on simulated data
is displayed in Figure~\ref{fig:uniTest} for balanced and for
imbalanced samples averaged over 1000 replications.  
In the balanced case, we simulate
$m_1=m_2=100$ Gaussian random variates with distributions
$\distNormal{0}{1}$ and $\distNormal{\mu}{1}$ for $\mu\in[0,1]$.
We compare the classic student's t-test to the permutation test
with 1000 permutations, the bounds from Theorem~\ref{thm:uniReal}
with $\kappa=1$, and the beta adjusted bound from 
Proposition~\ref{prop:betaAdj}. 
Notably, the balanced Khintchine bound returns 
p-values just slightly larger than the standard t-test while the 
beta adjusted bound is even tighter.
For the imbalanced case, the sample sizes are now $m_1=140,m_2=60$
and $\kappa=2.33$.
The imbalanced bound is not as sharp, but the beta adjusted 
bound still gives a close approximation to the t-test p-value.

\begin{figure}
	\begin{center}
		\includegraphics[width=0.4\textwidth]{\PICDIR/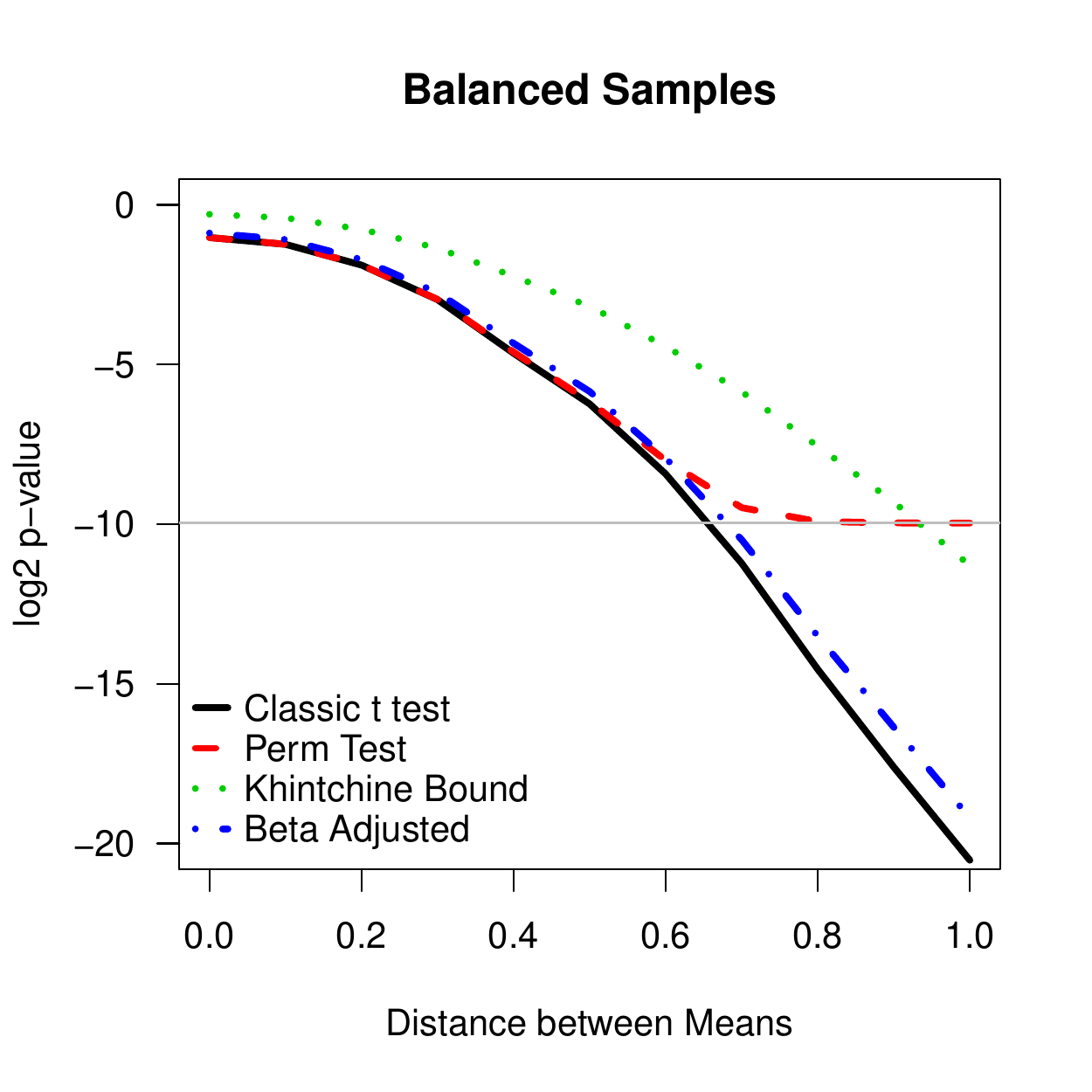}
		\includegraphics[width=0.4\textwidth]{\PICDIR/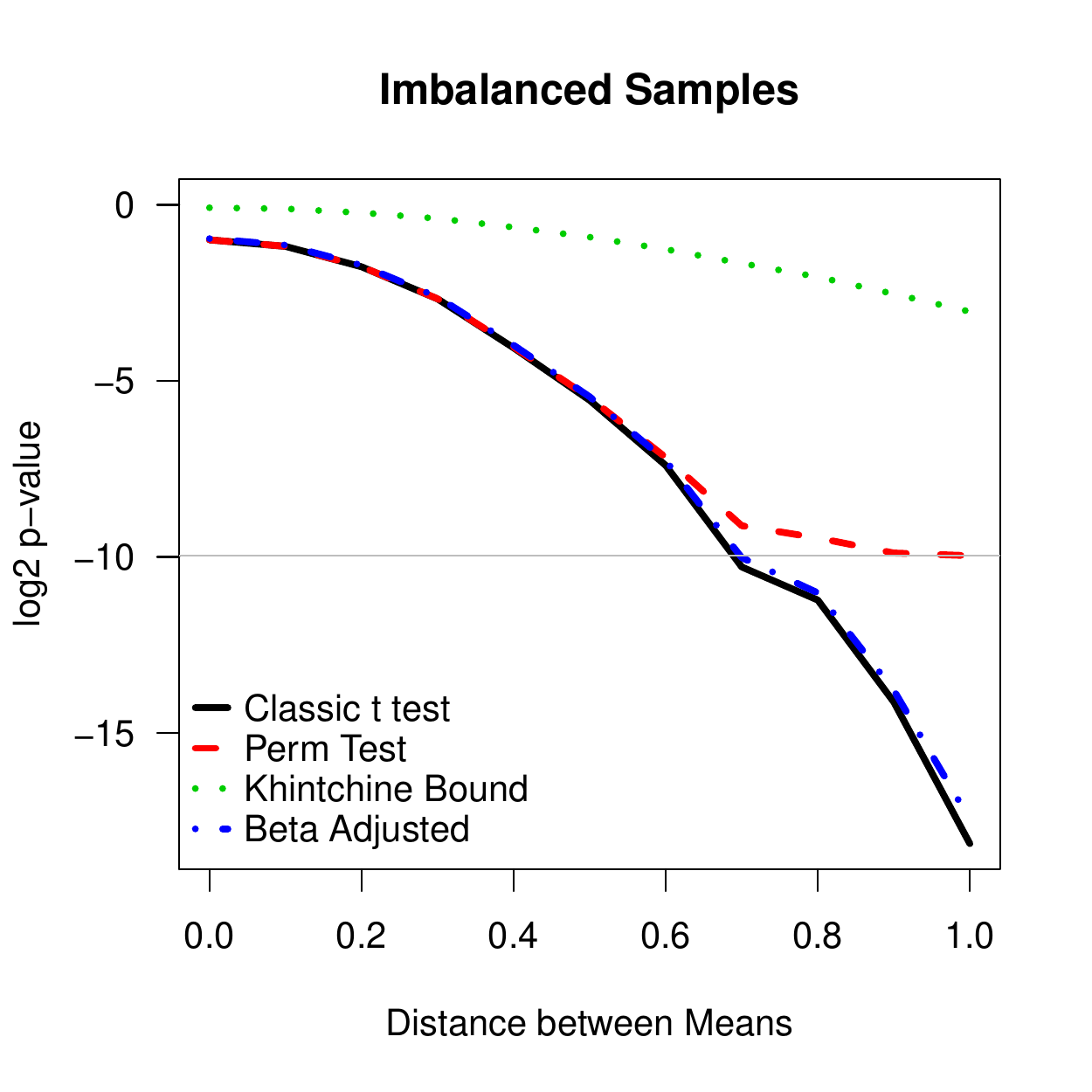}
	\end{center}
    \vspace{-0.25in}
	\caption{
		\label{fig:uniTest}
		Univariate two sample test for normal data with 
		balanced sample sizes $m_1=m_2=100$ (left)
		and for imbalanced $m_1=140,m_2=60$ (right)
		comparing the standard t-test (black)
		to the permutation test (red) with 1000 permutations and
		to the Khintchine bound, Theorem~\ref{thm:uniReal}, (green) 
		and the beta adjusted bound, Proposition~\ref{prop:betaAdj}, (blue) 
		all across 1000 replications.  
	}
\end{figure}

\subsubsection{K Sample Test}

The performance of Theorem~\ref{thm:syncTest}
for comparing $k$ samples of size $n$ via a synchronized permutation
test is demonstrated in Figure~\ref{fig:syncTest}.  For this 
simulation, $k=4,16$ for the left and right plot, respectively,
samples of size $n=20$ were generated as random Gaussian variates
with variance 1 and with mean 0 for the first $k-1$ sets and 
with mean $\mu\in[0,2]$ for the $k$th set.  As $\mu$ grows,
the p-value for the standard F-test, the synchronized permutation 
test, and the beta-adjusted p-value from Theorem~\ref{thm:syncTest}
all decrease in tandem for $k=4$ with the unadjusted bound above the 
others.  In the $k=16$ case, the beta adjusted
bound and the synchronized permutation test return the same p-values
until the lines approach the permutation boundary at $-\log_2(1001)$.
More notably, they slightly differ from the classic F-test as 
for relatively large $k$ and small $n$ the synchronized permutation
test returns marginally different p-values than the F-test.  
A total of 1000 random permutations were generated 
for the synchronized permutation test, and this simulation 
was replicated 1000 times to create these plots.

\begin{figure}
	\begin{center}
		\includegraphics[width=0.4\textwidth]{\PICDIR/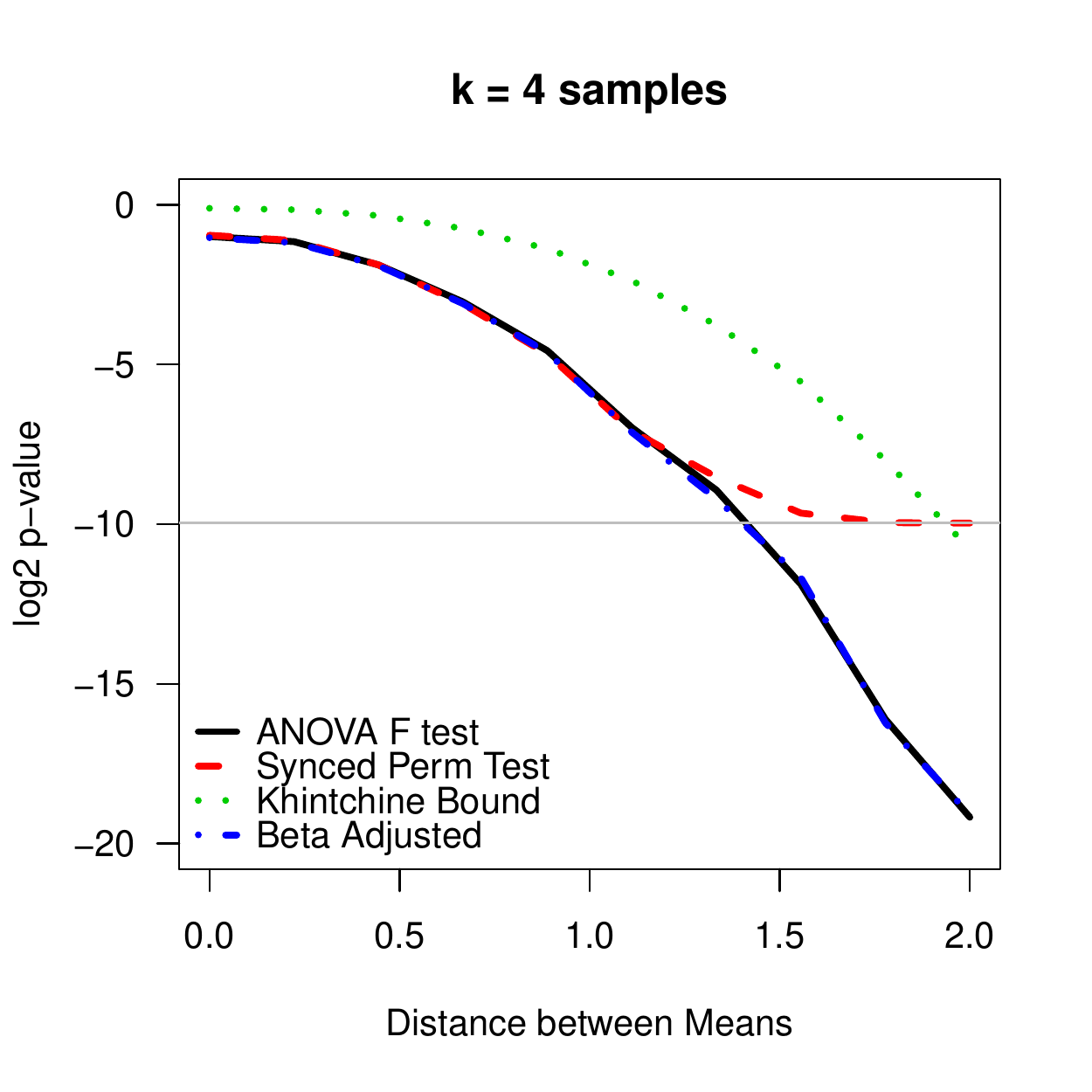}
		\includegraphics[width=0.4\textwidth]{\PICDIR/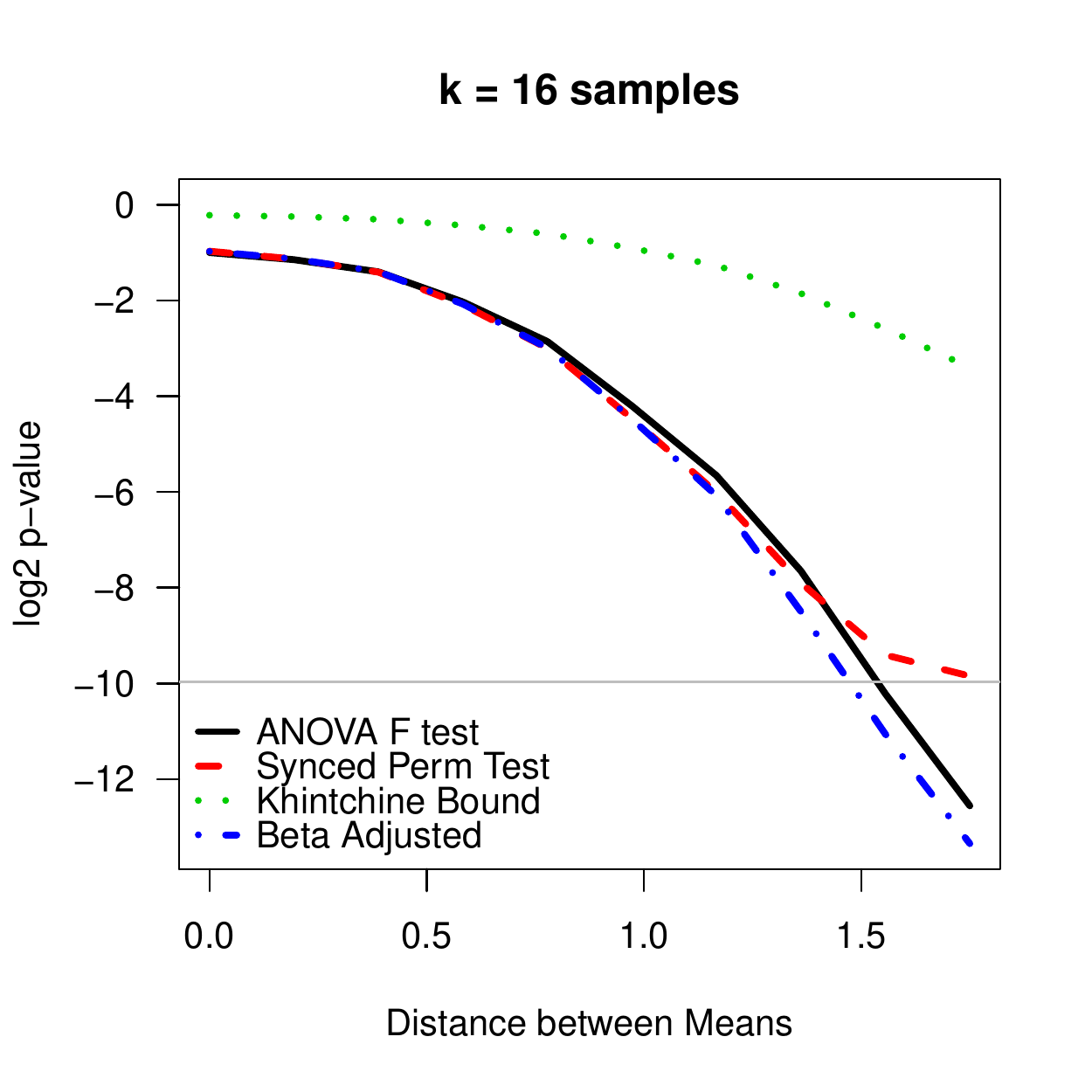}
	\end{center}
    \vspace{-0.25in}
	\caption{
		\label{fig:syncTest}
		Univariate $k$ sample test for normal data with 
		$k=4$ (left) and $k=16$ (right)
		balanced samples of size $n=20$.
		The figure compares the standard F-test (black)
		to the synchronized permutation test (red) with 1000 
		permutations and
		to the unadjusted (green) and beta adjusted (blue) 
		bounds from Theorem~\ref{thm:syncTest}.  
	}
\end{figure}

\subsection{Berkeley Growth Curves: Functional Means}
\label{eg:berkeleyCurves}
To demonstrate Theorem~\ref{thm:uniComm}, we apply it to the classic 
Berkeley growth curve dataset \citep{RAMSAYSILVERMAN2005}.\footnote{
	This data is available in the R package \texttt{fda} \citep{FDAPACK}.  
}  
This dataset
contains measurements of 93 children---39 males and 54 
females---taken at 31 time points between the ages of 1 and 18 years. 
A set of 30 curves was randomly selected from the male curves and 
30 curves from the female curves to test for a difference in the 
population mean curves based on those observations.  This was 
repeated 100 times to see the resulting p-values under the 
$L^1$, $L^2$, and $L^\infty$ norms. Table~\ref{tab:functMean} 
displays the percentage of rejections.  Applying 
Theorem~\ref{thm:uniComm} results in a reasonable number 
of rejections under the $L^1$ topology.  However, 
differences are not detectable in $L^2$ or $L^\infty$.
This is rectified via the empirical beta adjustment.

\begin{table}
	\caption{
		\label{tab:functMean}
		Displayed above are the percentages of rejections 
		by using Theorem~\ref{thm:uniComm} (left) and the 
		beta adjusted bound (right) at test sizes 
		$5\%,1\%,0.1\%$ 
		for the 
		$L^1$, $L^2$, and $L^\infty$ norms.  
	}
	\begin{center}
		\begin{tabular}{lrrrrrr}
			\hline
			&\multicolumn{6}{c}{Percentage of Rejections}\\
			&\multicolumn{3}{c}{Kahane Bound}&\multicolumn{3}{c}{Beta Adjusted}\\
			Size       & 5\% & 1\%  & 0.1\% &  5\% & 1\%  & 0.1\% \\\hline
			$L^1$       & 86\% & 42\%  &  7\% & 85\% & 55\% & 31\% \\
			$L^2$       & 0\%  & 0\%   &  0\% & 100\%& 88\% & 77\% \\
			$L^\infty$  & 0\%  & 0\%   &  0\% & 100\%& 100\%& 98\%\\
			\hline
		\end{tabular}
	\end{center}
\end{table}

\subsection{Phoneme Curves: Covariance Operators}
\label{eg:phonemeOperators}
We apply Theorem~\ref{thm:uniNonComm} to the 
classic phoneme dataset \citep{FERRATYVIEU2006},
which consists of 400 log-periodograms for 5 different
phonemes---the vowel from `dark' \texttt{aa}, 
the vowel from `water' \texttt{ao}, 
the plosive d-sound \texttt{dcl}, 
the fricative sh-sound \texttt{sh}, 
the vowel from she \texttt{iy}---sampled at 150 
frequencies.\footnote{
	This data is available in the R package \texttt{fds} \citep{FDSPACK}.  
}
Using the notation of the International Phonetic Alphabet (IPA),
\texttt{aa} is \textipa{A},
\texttt{ao} is \textipa{O},
\texttt{dcl} is \textipa{d},
\texttt{sh} is \textipa{S}, and
\texttt{iy} is \textipa{i}.
To produce covariance operators for testing, we first 
randomly permute the order of the 400 curves, then
group these curves into sets of 10 to produce a 
set of 40 covariance operators for each of the 
five phoneme classes. This is replicated 100 times
with different random groupings of curves.

We apply our method after using the empirical 
beta adjustment from Section~\ref{sec:betaAdj} 
to each of the 10 pairwise comparisons
between phonemes resulting in Table~\ref{tab:phmData}.
In the trace norm topology, all $100\times10$ pairwise
tests result in rejection for a test size of 1\%.  The 
Hilbert-Schmidt norm only detects a significant difference
between \textipa{A} and \textipa{O} about 52\% of the time
whereas the operator norm fails to detect any significance
between those two phonemes.  The difference between phonemes
\textipa{A} and \textipa{O} is hardest to identify among the
10 pairings.  

\begin{table}
	\caption{
		\label{tab:phmData}
		The percentage of rejected two sample tests at the 
		1\% level comparing two
		different phonemes with a sample size of 
		$m_1=m_2=10$ under the trace, Hilbert-Schmidt, 
		and operator norms.
	}
	\begin{center}
			\begin{tabular}{crrrrcrrrrcrrrr}
				\hline
				& \multicolumn{4}{c}{Trace Norm} &&\multicolumn{4}{c}{Hilbert-Schmidt Norm} &&
				\multicolumn{4}{c}{Operator Norm} \\
				& \textipa{A} &  \textipa{O} & \textipa{d} & \textipa{S} &\phantom{X}
				& \textipa{A} &  \textipa{O} & \textipa{d} & \textipa{S} &\phantom{X}
				& \textipa{A} &  \textipa{O} & \textipa{d} & \textipa{S} \\\hline  
				\textipa{O} &100&   &   &   && 52&   &   &   &&  0&   &   &    \\
				\textipa{d} &100&100&   &   && 93& 86&   &   && 15& 23&   &    \\
				\textipa{S} &100&100&100&   && 99&100&100&   && 88& 97& 91&    \\
				\textipa{i} &100&100&100&100&&100&100&100&100&&100&100&100&100\\
				\hline
		\end{tabular}
	\end{center}
\end{table}

\section{Phonological differences between vowels}
\label{sec:vowelData}

Taking inspiration from the classic phoneme dataset
\citep{HASTIE1995PDA,FERRATYVIEU2006} discussed previously
in Section~\ref{eg:phonemeOperators}, we consider a new data set of
log-periodograms for the phonemes of 12 spoken vowels 
detailed in Table~\ref{tab:vowels}.\footnote{Note that 
	this data
	was collected outside of a proper laboratory setting
	to be a proof-of-concept 
	for the proposed methodology as opposed to an in depth 
	study of language.}
This data is available at 
\url{https://sites.ualberta.ca/~kashlak/kashData.html}.

The raw data consists of 12 phonemes recorded 12 times 
each from 4 different speakers.  The data was recorded 
on a Tascam DR-05 portable linear PCM audio recorder
as a mono 24-bit wave file sampled at 96 kHz, which is currently
considered \textit{high definition audio} in contrast to the
standard 16-bit 44.1 kHz audio on compact discs.
The primary vowel phoneme 
was extracted as a 170 millisecond
clip corresponding to $16384=2^{14}$ samples.  These clips
were transformed into log periodograms via the \texttt{tuneR}
package~\citep{TUNER} as displayed in Figure~\ref{fig:allVowels} for
a single speaker.
As is common with functional data, the raw log-periodograms 
were first smoothed.  In this case, cubic smoothing splines 
were used.  However,
many other smoothing methods can be and have been 
applied to functional data.

\begin{figure}
	\begin{center}
		\includegraphics[width=\textwidth]{\PICDIR/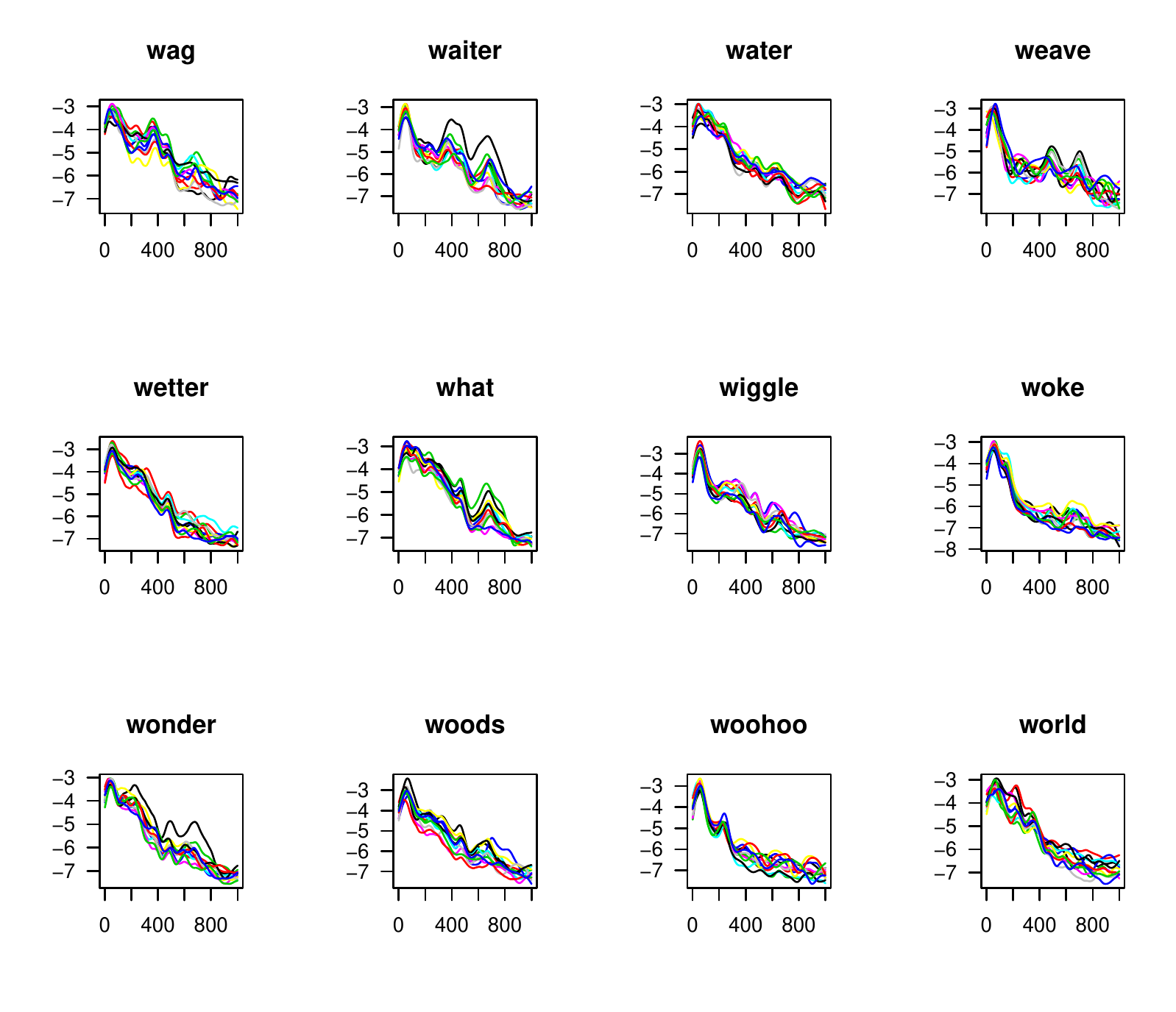}
	\end{center}
	\caption{
		\label{fig:allVowels}
		The log-periodograms of all 12 vowel phonemes spoken 12 times by 
		one of the speakers considered over the first 1000 frequencies.
	}
\end{figure}

Two experimental designs were employed in the collection of 
this data and will be tested in the following subsections.  
First, the 12 words were vocalized 12 times in a Latin
square design.  Each row corresponds to a replication of 
speaking all of the 12 words, and each column corresponds to 
the order of the
words within a replication.  This was done to test for changes in
speech during the recording period.
Secondly, this Latin square design was replicated for 4 different
speakers with two binary blocking factors male/female and 
Canadian/Chinese.  Thus, we have a 
$12\times4$ complete randomized block design with functional responses.
The total sample size is $576=12\times12\times4$ log-periodogram curves.

\begin{table}
	\caption{
		\label{tab:vowels}
		The twelve vowel phonemes considered in our dataset 
		along with the 12 spoken words used to produce those vowels.
	}
	\begin{center}
			\begin{tabular}{|cl|cl|cl|cl|}
				\hline
				\textipa{i}& weave  & \textipa{e}& waiter & 
				\textipa{E}& wetter & \textipa{\ae}& wag  \\
				\textipa{I}& wiggle & \textipa{9}& what   & 
				\textipa{u}& woohoo & \textipa{U}& woods\\
				\textipa{3}& world  & \textipa{o}& woke   & 
				\textipa{2}& wonder & \textipa{6}& water\\
				\hline
			\end{tabular}
	\end{center}
\end{table}

\begin{figure}
	\begin{center}
		\includegraphics[width=0.4\textwidth]{\PICDIR/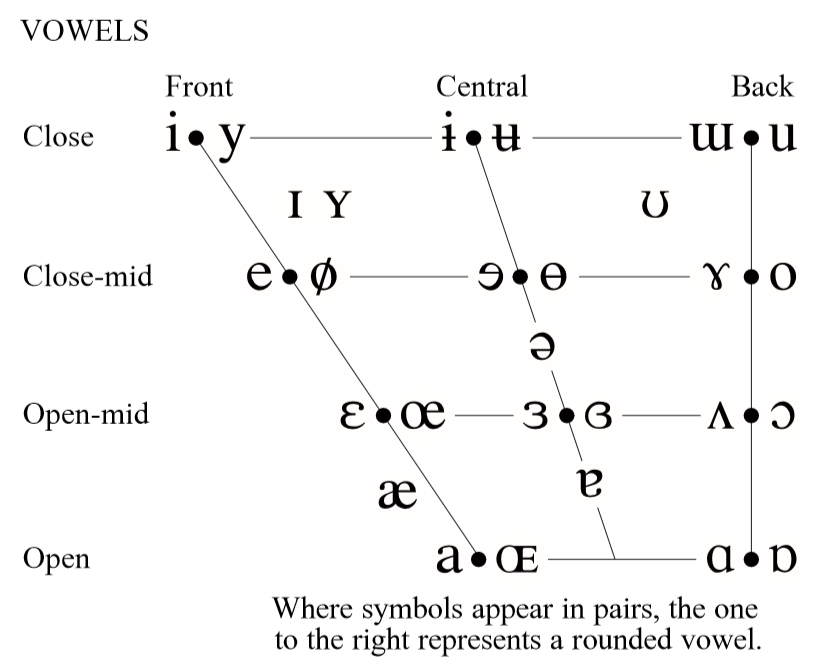}
	\end{center}
	\caption{
		\label{fig:ipaVowel}
		IPA Vowel Chart, \url{http://www.internationalphoneticassociation.org/content/ipa-chart}, available under a Creative Commons Attribution-Sharealike 3.0 Unported License. Copyright 2015 International Phonetic Association.
	}
\end{figure}

\subsection{Latin square design for functional means}

For an unreplicated Latin square design, we cannot perform 
a permutation test for the significance of each factor 
simultaneously.
Exchangability under the null hypothesis for one factor requires
fixing the levels of all other factors when permuting labels.
However, if we fix the Latin square row and column
indices then only a single observation remains leaving nothing to 
permute.  To rectify
this, a stepdown approach as in \cite{BASSO2009} chapter~7 
for unreplicated factorial designs can be applied.
As a permutation test requires exchangeable observations under
the null hypothesis, test statistics for each factor are first
computed.  Beginning with the largest, if that null hypothesis holds,
then this implies that all other null hypotheses hold and hence acts
as the global null allowing for the data to be permuted.
If this null is rejected, then we proceed to test the second largest
test statistic while fixing the levels of the first factor.  Once a 
null is not rejected, this method stops.  Otherwise, all factors can 
be tested except for the last one as rejecting all other null hypotheses
would leave no room for further permutations.

For the vowel data, we have 12-level row, column, and vowel factors 
giving the model
$
y_{ijk}(t) = \mu(t) + 
\text{row}_i(t) + \text{column}_j(t) + \text{vowel}_k(t) + \xi_{ijk}
$
where $y_{ijk}(t)$ is a smoothed log-periodogram, 
$\mu$ is the global mean, and the $\xi_{ijk}$ are mean zero
exchangeable errors.  For all four subjects, the vowel factor
produced a much larger test statistic than the row and
column effects as expected indicating consistency of the 
speaker during the experiment.  Thus, after rejecting
the null of there being no difference among the spoken vowels,
the row or column factor can be considered.  For all four subjects, 
the row and column effects were not deemed to be statistically 
significant---i.e. there were no detectable changes in pronunciation
across the recording session.  
Pairwise comparison of the vowels for each subject
was also performed.  However, of the 66 pairwise hypotheses to 
test, one subject rejected 25 nulls, another rejected only 5 nulls,
and the last two rejected 0 nulls after taking multiple testing
into account.  This is in contrast to the randomized block 
design discussed in the next section that, 
making use of the entire dataset and covariance operators, 
identifies 60 of the 66 pairings
as significantly different.

Before computing the test statistics and p-values in the 
randomized block design discussed next, each log-periodogram was 
centred by subtracting off the row and column effects from the
Latin square design.  This 
resulted in an improvement in the reported p-values, 
which were larger in the case that
the row and column effects were not removed.

\subsection{Complete Randomized block design for functional data}

A complete randomized block design (CRBD) aims to test a treatment effect 
as in one-way ANOVA 
but with the addition of blocking factors to account for sources of 
variation unrelated to the treatment of interest.
For functional data, a CRBD can be performed 
by using the synchronized permutation test for two-way ANOVA from 
chapter 6 of \cite{BASSO2009} combined with the Kahane-Khintchine 
based tail bound.  To achieve this, a difference between the 
functional means or covariances is computed for each of the
$12\choose2$ vowel pairings while holding the levels of the blocking
factors constant.  For each pairing, the test statistics can be 
summed over the levels of the blocking factors thus removing any 
influence from interaction terms even though they are generally assumed
to be negligible in this setting.  Theorems~\ref{thm:uniComm} 
and~\ref{thm:uniNonComm} can be applied for functional means and covariance
operators
respectively to bound the pairwise p-values.  The computed test 
statistics can be aggregated using Theorem~\ref{thm:syncTest} to
get a global p-value.
Note that a standard permutation test would require the computation
of $264=66\times4$ test statistics via simulation from 
the symmetric group, which in the case of covariance operators and 
Schatten norms implies 264 SVD calculations per permutation.
This is further expanded by, say, performing $132,000=66\times2000$
permutations to be able to test each hypothesis at the 0.01 level
after correcting for multiple testing.  Focusing only on the 
approximately 35 million required SVDs, a timing test run on 
an Intel Core i7-7567U CPU at 3.50GHz estimates 36 hours 
of compute time when considering $100\time100$ dimensional 
matrices and an estimated run time of 74 days on $400\times400$ 
dimensional matrices.

This approach was applied pairwise to the sample covariance operators
for each vowel as past work has emphasized that the covariance 
structure of speech data is the best lens to detect
phonological differences
\citep{PIGOLI2014,PIGOLI2018}.
Application of Theorem~\ref{thm:uniNonComm} using the 
trace norm (1-Schatten norm) and using the 
empirical beta adjustment from Section~\ref{sec:betaAdj}
produced the 66 pairwise p-values displayed in 
Figure~\ref{fig:vowelCorrplot}.  
The words are also grouped by
p-value to display which 
vowel phonemes proved statistically indistinguishable using
our proposed methodology.
The use of other Schatten norms results in lower power---i.e.
fewer null hypotheses rejected. 

The blocking factors \{male,female\} and \{Canadian, Chinese\}
can also be similarly tested without removing the row and column
effects from the Latin square design; otherwise, the mean taken over
the entire dataset will be zero.  
In trace norm, we get 
p-values of 0.0002 and 0.00003 for sex and country, respectively.
In Hilbert-Schmidt norm, we get the weaker p-values 0.03 and 0.07.

\begin{figure}
	\begin{center}
		\includegraphics[width=\textwidth]{\PICDIR/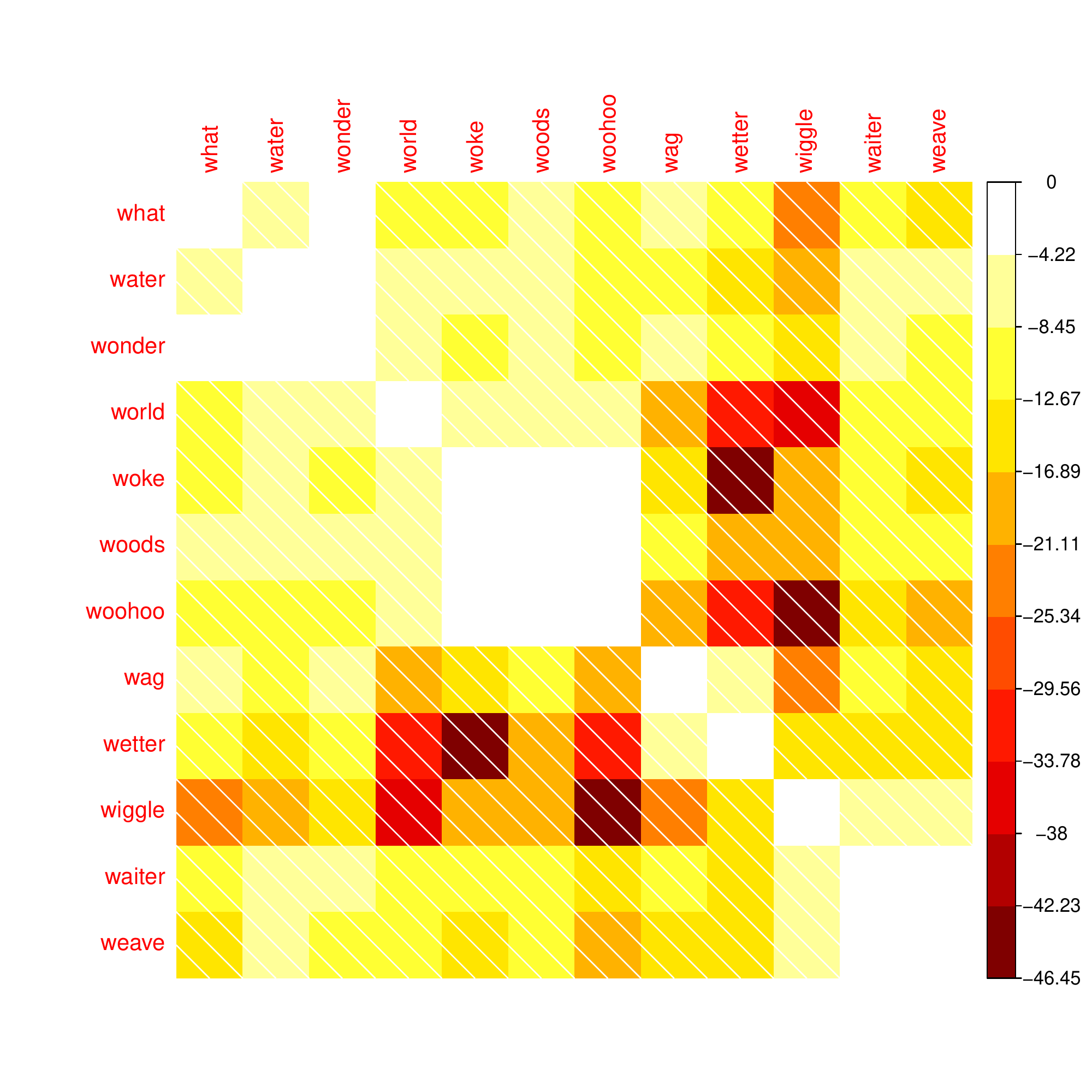}
	\end{center}
	\caption{
		\label{fig:vowelCorrplot}
		$\log_{10}$ p-values for pairwise two sample tests
		between vowel pairs under the trace norm.
	}
\end{figure}

\section{Discussion}

The p-value has stood for over a century as a pillar of 
frequentist statistical methodology.
In this article, we approached $k$-sample testing through 
application of an analytic approximation to the permutation test 
p-value notably 
without relying on simulation
of the permutation distribution of the test statistic.
Experimental design for functional data was the main 
motivation for this work as standard simulation-based
permutation testing can be applied but at a high computational
cost.  
Other applications of interest include online testing
where data must be processed, results returned, and decisions
made in real time.  The lag resulting from a classic permutation
test is unacceptable in such settings.
This methodology 
is generally applicable to other complex testing settings 
including other types of group invariances---e.g. 
rotationally invariant test statistics.
Furthermore, the duality of hypothesis testing with confidence
sets suggests investigation into using variants of 
the Kahane--Khintchine inequality to construct confidence balls for 
estimators with finite sample guarantees on the coverage.

\section*{Supplementary Material}

Primarily, the supplementary material contains proofs of the 
main results as well as auxilary theorems and lemmas.  
Secondly, it contains further exposition of this methodology
on simulated and real data sets.  In particular, these methods
are shown to produce correct p-values in the null setting.

\bibliographystyle{Chicago}

\def\BIBDIRWIN{.}
\bibliography{\BIBDIRWIN/kasharticle,\BIBDIRWIN/kashbook,\BIBDIRWIN/kashpack,\BIBDIRWIN/kashself}

\begin{thebibliography}{}

\bibitem[\protect\citeauthoryear{Agresti}{Agresti}{1992}]{AGRESTI1992}
Agresti, A. (1992).
\newblock A survey of exact inference for contingency tables.
\newblock {\em Statistical science\/}~{\em 7\/}(1), 131--153.

\bibitem[\protect\citeauthoryear{Basso, Pesarin, Salmaso, and Solari}{Basso
  et~al.}{2009}]{BASSO2009}
Basso, D., F.~Pesarin, L.~Salmaso, and A.~Solari (2009).
\newblock {\em Permutation tests for stochastic ordering and ANOVA: theory and
  applications with R}, Volume 194.
\newblock Springer Science \& Business Media.

\bibitem[\protect\citeauthoryear{Boucheron, Lugosi, and Massart}{Boucheron
  et~al.}{2013}]{BOUCHERON2013}
Boucheron, S., G.~Lugosi, and P.~Massart (2013).
\newblock {\em Concentration inequalities: A nonasymptotic theory of
  independence}.
\newblock Oxford University Press.

\bibitem[\protect\citeauthoryear{Bousquet}{Bousquet}{2003}]{BOUSQUET2003}
Bousquet, O. (2003).
\newblock Concentration inequalities for sub-additive functions using the
  entropy method.
\newblock In {\em Stochastic inequalities and applications}, pp.\  213--247.
  Springer.

\bibitem[\protect\citeauthoryear{Brombin and Salmaso}{Brombin and
  Salmaso}{2013}]{BROMBIN2013}
Brombin, C. and L.~Salmaso (2013).
\newblock {\em Permutation tests in shape analysis}, Volume~15.
\newblock Springer.

\bibitem[\protect\citeauthoryear{Burak and Kashlak}{Burak and
  Kashlak}{2021}]{KASHLAK_BURAK_WILDBS}
Burak, K.~L. and A.~B. Kashlak (2021).
\newblock Nonparametric confidence regions via the analytic wild bootstrap.
\newblock {\em (under review)\/}.

\bibitem[\protect\citeauthoryear{Cabassi, Pigoli, Secchi, and Carter}{Cabassi
  et~al.}{2017}]{CABASSI2017}
Cabassi, A., D.~Pigoli, P.~Secchi, and P.~A. Carter (2017).
\newblock Permutation tests for the equality of covariance operators of
  functional data with applications to evolutionary biology.
\newblock {\em Electronic Journal of Statistics\/}~{\em 11\/}(2), 3815--3840.

\bibitem[\protect\citeauthoryear{Chakraborty and Chaudhuri}{Chakraborty and
  Chaudhuri}{2015}]{CHAKRABORTY2015}
Chakraborty, A. and P.~Chaudhuri (2015).
\newblock A wilcoxon--mann--whitney-type test for infinite-dimensional data.
\newblock {\em Biometrika\/}~{\em 102\/}(1), 239--246.

\bibitem[\protect\citeauthoryear{Corain, Melas, Pepelyshev, and Salmaso}{Corain
  et~al.}{2014}]{CORAIN2014}
Corain, L., V.~B. Melas, A.~Pepelyshev, and L.~Salmaso (2014).
\newblock New insights on permutation approach for hypothesis testing on
  functional data.
\newblock {\em Advances in Data Analysis and Classification\/}~{\em 8\/}(3),
  339--356.

\bibitem[\protect\citeauthoryear{Cox and Lee}{Cox and Lee}{2008}]{COX2008}
Cox, D.~D. and J.~S. Lee (2008).
\newblock Pointwise testing with functional data using the westfall--young
  randomization method.
\newblock {\em Biometrika\/}~{\em 95\/}(3), 621--634.

\bibitem[\protect\citeauthoryear{De~la Pena and Gin{\'e}}{De~la Pena and
  Gin{\'e}}{2012}]{DELAPENAGINE2012}
De~la Pena, V. and E.~Gin{\'e} (2012).
\newblock {\em Decoupling: from dependence to independence}.
\newblock Springer Science \& Business Media.

\bibitem[\protect\citeauthoryear{Diestel, Jarchow, and Tonge}{Diestel
  et~al.}{1995}]{DIESTEL1995}
Diestel, J., H.~Jarchow, and A.~Tonge (1995).
\newblock {\em Absolutely summing operators}, Volume~43.
\newblock Cambridge university press.

\bibitem[\protect\citeauthoryear{Ferraty and Vieu}{Ferraty and
  Vieu}{2006}]{FERRATYVIEU2006}
Ferraty, F. and P.~Vieu (2006).
\newblock {\em Nonparametric functional data analysis: theory and practice}.
\newblock Springer Science \& Business Media.

\bibitem[\protect\citeauthoryear{Garling}{Garling}{2007}]{GARLING2007}
Garling, D.~J. (2007).
\newblock {\em Inequalities: a journey into linear analysis}.
\newblock Cambridge University Press.

\bibitem[\protect\citeauthoryear{Gin{\'e} and Nickl}{Gin{\'e} and
  Nickl}{2016}]{GINENICKL2015}
Gin{\'e}, E. and R.~Nickl (2016).
\newblock {\em Mathematical Foundations of Infinite-Dimensional Statistical
  Models}.
\newblock Cambridge University Press.

\bibitem[\protect\citeauthoryear{Good}{Good}{1956}]{GOOD1956}
Good, I. (1956).
\newblock On the estimation of small frequencies in contingency tables.
\newblock {\em Journal of the Royal Statistical Society: Series B
  (Methodological)\/}~{\em 18\/}(1), 113--124.

\bibitem[\protect\citeauthoryear{Good}{Good}{2013}]{GOOD2013}
Good, P. (2013).
\newblock {\em Permutation tests: a practical guide to resampling methods for
  testing hypotheses}.
\newblock Springer Science \& Business Media.

\bibitem[\protect\citeauthoryear{Gretton, Borgwardt, Rasch, Sch{\"o}lkopf, and
  Smola}{Gretton et~al.}{2012}]{GRETTON2012}
Gretton, A., K.~M. Borgwardt, M.~J. Rasch, B.~Sch{\"o}lkopf, and A.~Smola
  (2012).
\newblock A kernel two-sample test.
\newblock {\em Journal of Machine Learning Research\/}~{\em 13\/}(Mar),
  723--773.

\bibitem[\protect\citeauthoryear{Haagerup}{Haagerup}{1981}]{HAAGERUP1981}
Haagerup, U. (1981).
\newblock The best constants in the {Khintchine} inequality.
\newblock {\em Studia Mathematica\/}~{\em 70}, 231--283.

\bibitem[\protect\citeauthoryear{Hall}{Hall}{2013}]{HALL2013}
Hall, P. (2013).
\newblock {\em The bootstrap and Edgeworth expansion}.
\newblock Springer Science \& Business Media.

\bibitem[\protect\citeauthoryear{Hastie, Buja, and Tibshirani}{Hastie
  et~al.}{1995}]{HASTIE1995PDA}
Hastie, T., A.~Buja, and R.~Tibshirani (1995).
\newblock Penalized discriminant analysis.
\newblock {\em The Annals of Statistics\/}, 73--102.

\bibitem[\protect\citeauthoryear{He, Basu, Zhao, and Owen}{He
  et~al.}{2019}]{HE2019}
He, H.~Y., K.~Basu, Q.~Zhao, and A.~B. Owen (2019).
\newblock Permutation $ p $-value approximation via generalized stolarsky
  invariance.
\newblock {\em The Annals of Statistics\/}~{\em 47\/}(1), 583--611.

\bibitem[\protect\citeauthoryear{Hemerik and Goeman}{Hemerik and
  Goeman}{2018}]{HEMERIK2018}
Hemerik, J. and J.~Goeman (2018).
\newblock Exact testing with random permutations.
\newblock {\em Test\/}~{\em 27\/}(4), 811--825.

\bibitem[\protect\citeauthoryear{Holm}{Holm}{1979}]{HOLM1979}
Holm, S. (1979).
\newblock A simple sequentially rejective multiple test procedure.
\newblock {\em Scandinavian journal of statistics\/}, 65--70.

\bibitem[\protect\citeauthoryear{Kahane}{Kahane}{1964}]{KAHANE1964}
Kahane, J.-P. (1964).
\newblock Sur les sommes vectorielles sigma plus minus un.
\newblock {\em Comptes rendus hebdomadaires des seances de l'academie des
  sciences\/}~{\em 259\/}(16), 2577.

\bibitem[\protect\citeauthoryear{Kallenberg}{Kallenberg}{2006}]{KALLENBERG2006}
Kallenberg, O. (2006).
\newblock {\em Probabilistic symmetries and invariance principles}.
\newblock Springer Science \& Business Media.

\bibitem[\protect\citeauthoryear{Kashlak and Yuan}{Kashlak and
  Yuan}{2020}]{KASHLAK_YUAN_ABELECT}
Kashlak, A.~B. and W.~Yuan (2020).
\newblock Computation-free nonparametric testing for local and global spatial
  autocorrelation with application to the canadian electorate.
\newblock {\em arXiv preprint arXiv:2012.08647\/}.

\bibitem[\protect\citeauthoryear{Klein and Rio}{Klein and
  Rio}{2005}]{KLEINRIO2005}
Klein, T. and E.~Rio (2005).
\newblock Concentration around the mean for maxima of empirical processes.
\newblock {\em The Annals of Probability\/}~{\em 33\/}(3), 1060--1077.

\bibitem[\protect\citeauthoryear{Kwapien}{Kwapien}{1987}]{KWAPIEN1987}
Kwapien, S. (1987).
\newblock Decoupling inequalities for polynomial chaos.
\newblock {\em The Annals of Probability\/}~{\em 15\/}(3), 1062--1071.

\bibitem[\protect\citeauthoryear{Lata{\l}a and Oleszkiewicz}{Lata{\l}a and
  Oleszkiewicz}{1994}]{LATALA1994}
Lata{\l}a, R. and K.~Oleszkiewicz (1994).
\newblock On the best constant in the {Khintchine}--{Kahane} inequality.
\newblock {\em Studia Mathematica\/}~{\em 109\/}(1), 101--104.

\bibitem[\protect\citeauthoryear{Ledoux and Talagrand}{Ledoux and
  Talagrand}{1991}]{LEDOUXTALAGRAND1991}
Ledoux, M. and M.~Talagrand (1991).
\newblock {\em Probability in Banach Spaces: isoperimetry and processes},
  Volume~23.
\newblock Springer.

\bibitem[\protect\citeauthoryear{Ligges, Krey, Mersmann, and
  Schnackenberg}{Ligges et~al.}{2018}]{TUNER}
Ligges, U., S.~Krey, O.~Mersmann, and S.~Schnackenberg (2018).
\newblock {\em {tuneR}: Analysis of Music and Speech}.

\bibitem[\protect\citeauthoryear{Mielke and Berry}{Mielke and
  Berry}{2007}]{MIELKE2007}
Mielke, P.~W. and K.~J. Berry (2007).
\newblock {\em Permutation methods: a distance function approach}.
\newblock Springer Science \& Business Media.

\bibitem[\protect\citeauthoryear{Pesarin and Salmaso}{Pesarin and
  Salmaso}{2010}]{PESARIN2010}
Pesarin, F. and L.~Salmaso (2010).
\newblock {\em Permutation tests for complex data: theory, applications and
  software}.
\newblock John Wiley \& Sons.

\bibitem[\protect\citeauthoryear{Pigoli, Aston, Dryden, and Secchi}{Pigoli
  et~al.}{2014}]{PIGOLI2014}
Pigoli, D., J.~A. Aston, I.~L. Dryden, and P.~Secchi (2014).
\newblock Distances and inference for covariance operators.
\newblock {\em Biometrika\/}, asu008.

\bibitem[\protect\citeauthoryear{Pigoli, Hadjipantelis, Coleman, and
  Aston}{Pigoli et~al.}{2018}]{PIGOLI2018}
Pigoli, D., P.~Z. Hadjipantelis, J.~S. Coleman, and J.~A. Aston (2018).
\newblock The statistical analysis of acoustic phonetic data: exploring
  differences between spoken romance languages.
\newblock {\em Journal of the Royal Statistical Society: Series C (Applied
  Statistics)\/}~{\em 67\/}(5), 1103--1145.

\bibitem[\protect\citeauthoryear{Pisier and Xu}{Pisier and
  Xu}{2003}]{PISIERXU2003}
Pisier, G. and Q.~Xu (2003).
\newblock Non-commutative lp-spaces.
\newblock {\em Handbook of the geometry of Banach spaces\/}~{\em 2},
  1459--1517.

\bibitem[\protect\citeauthoryear{Ramsay and Silverman}{Ramsay and
  Silverman}{2005}]{RAMSAYSILVERMAN2005}
Ramsay, J.~O. and B.~W. Silverman (2005).
\newblock {\em Functional data analysis}.
\newblock New York: Springer.

\bibitem[\protect\citeauthoryear{Ramsay, Wickham, Graves, and Hooker}{Ramsay
  et~al.}{2018}]{FDAPACK}
Ramsay, J.~O., H.~Wickham, S.~Graves, and G.~Hooker (2018).
\newblock {\em fda: Functional Data Analysis}.
\newblock R package version 2.4.8.

\bibitem[\protect\citeauthoryear{Romano and Wolf}{Romano and
  Wolf}{2005}]{ROMANO2005}
Romano, J.~P. and M.~Wolf (2005).
\newblock Exact and approximate stepdown methods for multiple hypothesis
  testing.
\newblock {\em Journal of the American Statistical Association\/}~{\em
  100\/}(469), 94--108.

\bibitem[\protect\citeauthoryear{Segal, Braun, Elliott, and Jiang}{Segal
  et~al.}{2018}]{SEGAL2018}
Segal, B.~D., T.~Braun, M.~R. Elliott, and H.~Jiang (2018).
\newblock Fast approximation of small p-values in permutation tests by
  partitioning the permutations.
\newblock {\em Biometrics\/}~{\em 74\/}(1), 196--206.

\bibitem[\protect\citeauthoryear{Shang and Hyndman}{Shang and
  Hyndman}{2013}]{FDSPACK}
Shang, H.~L. and R.~J. Hyndman (2013).
\newblock {\em fds: Functional data sets}.
\newblock R package version 1.7.

\bibitem[\protect\citeauthoryear{Solomon and Stephens}{Solomon and
  Stephens}{1978}]{SOLOMON1978}
Solomon, H. and M.~A. Stephens (1978).
\newblock Approximations to density functions using pearson curves.
\newblock {\em Journal of the American Statistical Association\/}~{\em
  73\/}(361), 153--160.

\bibitem[\protect\citeauthoryear{Spektor}{Spektor}{2014}]{SPEKTOR_THESIS}
Spektor, S. (2014).
\newblock {\em Selected Topics in Asymptotic Geometric Analysis and
  Approximation Theory}.
\newblock Ph.\ D. thesis, University of Alberta.

\bibitem[\protect\citeauthoryear{Spektor}{Spektor}{2016}]{SPEKTOR2016}
Spektor, S. (2016).
\newblock Restricted {Khinchine} inequality.
\newblock {\em Canadian Mathematical Bulletin\/}~{\em 59\/}(1), 204--210.

\bibitem[\protect\citeauthoryear{Stuart, Arnold, Ord, O'Hagan, and
  Forster}{Stuart et~al.}{1994}]{KENDALL}
Stuart, A., S.~Arnold, J.~K. Ord, A.~O'Hagan, and J.~Forster (1994).
\newblock {\em Kendall's advanced theory of statistics}.
\newblock Wiley.

\bibitem[\protect\citeauthoryear{Talagrand}{Talagrand}{1996}]{TALAGRAND1996CON}
Talagrand, M. (1996).
\newblock New concentration inequalities in product spaces.
\newblock {\em Inventiones mathematicae\/}~{\em 126\/}(3), 505--563.

\bibitem[\protect\citeauthoryear{Watson}{Watson}{1959}]{WATSON1959}
Watson, G. (1959).
\newblock A note on gamma functions.
\newblock {\em Edinburgh Mathematical Notes\/}~{\em 42}, 7--9.

\bibitem[\protect\citeauthoryear{Westfall and Young}{Westfall and
  Young}{1993}]{WESTFALL_YOUNG}
Westfall, P.~H. and S.~S. Young (1993).
\newblock {\em Resampling-based multiple testing: Examples and methods for
  p-value adjustment}, Volume 279.
\newblock John Wiley \& Sons.

\bibitem[\protect\citeauthoryear{Winkler, Ridgway, Douaud, Nichols, and
  Smith}{Winkler et~al.}{2016}]{WINKLER2016}
Winkler, A.~M., G.~R. Ridgway, G.~Douaud, T.~E. Nichols, and S.~M. Smith
  (2016).
\newblock Faster permutation inference in brain imaging.
\newblock {\em Neuroimage\/}~{\em 141}, 502--516.

\bibitem[\protect\citeauthoryear{Wu and Hamada}{Wu and Hamada}{2011}]{WUHAMADA}
Wu, C.~J. and M.~S. Hamada (2011).
\newblock {\em Experiments: planning, analysis, and optimization}, Volume 552.
\newblock John Wiley \& Sons.

\bibitem[\protect\citeauthoryear{Yang, Trucco, and Buu}{Yang
  et~al.}{2019}]{YANG2019}
Yang, J.~J., E.~M. Trucco, and A.~Buu (2019).
\newblock A hybrid method of the sequential monte carlo and the edgeworth
  expansion for computation of very small p-values in permutation tests.
\newblock {\em Statistical methods in medical research\/}~{\em 28\/}(10-11),
  2937--2951.

\end{thebibliography}

\appendix

\section{Inequalities}
\label{app:ineq}
\subsection{Khintchine-type Inequalities}

\begin{thm}[Khintchine's Inequality (1923)]
	\label{thm:classicKhintchine}
	For any $p\in(0,\infty)$, there exist positive 
	finite constants $A_p$ and $B_p$ such that for 
	any sequence $x_1,\ldots,x_n\in\real$ (or $x_i\in\complex$),
	$$
	A_p^p\norm{x}_{\ell^2}^p \le \xv\abs*{\sum_{i=1}^n 
		\veps_i x_i}^p
	\le B_p^p\norm{x}_{\ell^2}^p
	$$
	where $\veps_1,\ldots,\veps_n$ are iid Rademacher random variables---i.e. $\prob{\veps_i=1}=\prob{\veps_i=-1}=1/2$.
\end{thm}

For this article, we are only concerned with the upper bound
$B_{2p}$ for $p>2$. In \cite{GARLING2007}, 
$B_{2p} = [(2p)!/2^pp!]^{1/2p}$ which gives
$B_{2p}<\sqrt{2p}$, 
but also via Stirling's inequality $B_p\sim(p/e)^{1/2}$
as $p\rightarrow\infty$.
The expectation in above theorem is with respect to the 
$\veps_i$ corresponding to a uniform distribution on the 
$2^n$ vertices of the $n$-hypercube.  In what follows,
we consider expectation over the uniform distribution
on the $n!$ elements of the symmetric group $\mathbb{S}_n$.
This will be denoted $\xv_\pi$ where $\pi\in\mathbb{S}_n$
is treated as a uniform random permutation.

In \cite{SPEKTOR2016}, the restricted Khintchine inequality
is introduced where it is required that $\sum_{i=1}^n \veps_i=0$
introducing a weak dependency among the $\veps_i$.  
In the proof in \cite{SPEKTOR2016}, this weak dependency 
doubles the variance by comparing two 
sets of data. Thus, the constant
becomes $B_{2p} = [(2p)!/p!]^{1/2p}$.

\begin{thm}[\cite{SPEKTOR2016} Theorem 1.1]
	\label{thm:conditionalKhintchine}
	For any $p\in[2,\infty)$, there exist positive 
	finite constant $B_p$ such that for 
	any sequence $x_1,\ldots,x_n\in\real$,
	\begin{equation}
	\label{eqn:susannaTheorem}
	\xv\abs*{\sum_{i=1}^n \veps_i x_i}^p
	\le B_p^p\left( 
	\norm{x}_{\ell^2}^2-n\bar{x}^2
	\right)^{p/2} 
	= B_p^p[(n-1)s_n^2]^{p/2}
	\end{equation}
	where $\veps_1,\ldots,\veps_n$ are Rademacher random variables
	such that $\sum \veps_i=0$ and 
	$s_n^2 = (n-1)^{-1}\sum_{i=1}^n(x_i-\bar{x})^2$ 
	is the sample variance of $x$.
\end{thm}

\begin{remark}
	In the statistics context, if we divide Inequality~\ref{eqn:susannaTheorem} 
	by $m=n/2$, we have
	$
	\xv_\pi\abs{ \bar{x}^{(\pi)}_1 - \bar{x}^{(\pi)}_2 }^p
	\le B_p(2s_n^2/m)^{p/2}
	$
	where $\bar{x}^{(\pi)}_1$ is the average of the first $m$ of the
	$x_{\pi(i)}$ for some random permutation $\pi$ and similarly
	for $\bar{x}^{(\pi)}_2$.
\end{remark}

The previous theorem only applies to a balanced two sample setting.
In the following, we extend the ideas in \cite{SPEKTOR2016} to 
the imbalanced testing setting.  Other such extensions to 
imbalanced Khintchine inequalities were considered in 
\cite{SPEKTOR_THESIS}.  Note that in the following theorem,
the bound on the right-hand-side is in terms of the smaller of 
the two sample sizes
$m_2 < m_1$ reducing the power drastically in a highly 
imbalanced setting.  Nevertheless, it still is seen to 
be an excellent statistical tool in a variety of applied 
settings after the beta correction
is applied \citep{KASHLAK_YUAN_ABELECT}.

\begin{thm}[Imbalanced Case]
	\label{thm:unbalancedKhintchine}
	For $m_1>m_2>0$, let $n = m_1+m_2$ and $M = m_1-m_2$
	and let $\kappa m_2 = m_1$ for some rational $\kappa>1$.
	Let $\delta_1,\ldots,\delta_n$ 
	be weighted dependent Rademacher random variables 
	such that marginally 
	$\prob{\delta_i=1/m_1} = \prob{\delta_i=-1/m_2}=1/2$
	and such that
	$\sum \delta_i=0$---i.e. precisely $m_1$ of the $\delta_i$ equal
	$1/m_1$ and $m_2$ equal $-1/m_2$.
	For any $p\in[2,\infty)$, 
	there exists a positive 
	finite constant $B_p$ such that for 
	any sequence $x_1,\ldots,x_n\in\real$,\footnote{This theorem
		is also valid for $x_i\in\complex$ after standard alterations
		are made in the proof.}
	$$
	\xv\abs*{\sum_{i=1}^n \delta_i x_i}^p
	\le 
	B_p\left(\frac{ \lceil\kappa+1\rceil^2{s}_{n}^2}{2m_2}\right)^{p/2}
	$$
	where 
	$s_n^2 = (n-1)^{-1}\sum_{i=1}^n(x_i-\bar{x})^2$ 
	is the sample variance of $x$.
\end{thm}

\begin{lemma}
	\label{lem:convexity}
	Let $\xi_1,\xi_2\in[1,\infty]$ be such that $\xi_1^{-1}+\xi_2^{-1}=1$,
	and let $X,Y$ be positive real random variables.  Then,
	$$
	\min_{
		\xi_1,\xi_2:\, \xi_1^{-1}+\xi_2^{-1}=1
	}\xv  \left\{ \xi_1^{p-1}X^p + \xi_2^{p-1}Y^p \right\}
	= \left[ (\xv X^p)^{1/p} + (\xv Y^p)^{1/p} \right]^p.
	$$
\end{lemma}
\begin{proof}
	We note that $\xi_2 = \xi_1/(\xi_1-1)$.  Then, 
	\begin{align*}
	0     &= \frac{d}{d\xi_1} \left\{\xi_1^{p-1}\xv X^p + \xi_2^{p-1}\xv Y^p\right\}\\
	&= (p-1)\xi_1^{p-2}\xv X^p - (p-1)\xi_1^{p-2}\xv Y^p/(\xi_1-1)^p\\
	\xi_1 &= 1 + \left( \xv Y^p/\xv X^p \right)^{1/p}\\
	\xi_2 &= 1 + \left( \xv X^p/\xv Y^p \right)^{1/p}
	\end{align*}
	Hence, 
	\begin{align*}
	&\min_{
		\xi_1,\xi_2:\, \xi_1^{-1}+\xi_2^{-1}=1
	}\xv  \left\{ \xi_1^{p-1}X^p + \xi_2^{p-1}Y^p \right\}\\
	&= \left[1+\left(\frac{\xv Y^p}{\xv X^p}\right)^{1/p}\right]^{p-1}\xv X^p + 
	\left[1+\left(\frac{\xv X^p}{\xv Y^p}\right)^{1/p}\right]^{p-1}\xv Y^p\\
	&= \left[(\xv Y^p)^{1/p}+(\xv X^p)^{1/p}\right]^{p-1}(\xv X^p)^{1/p} + 
	\left[(\xv X^p)^{1/p}+(\xv Y^p)^{1/p}\right]^{p-1}(\xv Y^p)^{1/p}\\
	&= [ (\xv X^p)^{1/p} + (\xv Y^p)^{1/p} ]^p 
	\end{align*}
\end{proof}
\begin{proof}[Proof of Theorem~\ref{thm:unbalancedKhintchine}]
	We first decompose the weighted Rademacher sum.  Without
	loss of generality, assume $m_1 > m_2$ and let $n=m_1+m_2$
	and $M = m_1-m_2$.  Also, assume the $x_i$ are centred---i.e.
	$\sum_{i=1}^nx_i=0$---and let $\xi_1,\xi_2>0$ such that 
	$\xi_1^{-1}+\xi_2^{-1}=1$.  Thus, via convexity, we have
	\begin{align*}
	\xv&\abs*{\sum_{i=1}^n \delta_i x_i}^p\\
	&= \xv_\pi\abs*{ \frac{1}{m_2}\sum_{i=1}^{m_2} x_{\pi(i)} - 
		\frac{1}{m_1}\sum_{i=m_2+1}^n x_{\pi(i)}  }^p\\
	&= \xv_\pi\abs*{ 
		\frac{1}{m_2}\left\{
		\sum_{i=1}^{m_2} x_{\pi(i)} - \sum_{i=m_2+1}^{2m_2} x_{\pi(i)}
		\right\} - \frac{1}{m_1}\sum_{i=2m_2+1}^n x_{\pi(i)}
		+ \frac{M}{m_1m_2}\sum_{i=m_2+1}^{2m_2} x_{\pi(i)}
	}^p\\
	&\le \frac{\xi_1^{p-1}}{m_2^p} \xv_\pi\abs*{ 
		\sum_{i=1}^{m_2} x_{\pi(i)} - \sum_{i=m_2+1}^{2m_2} x_{\pi(i)}
	}^p + \frac{\xi_2^{p-1}M^p}{m_1^p}\xv_\pi\abs*{
		\frac{1}{M}\sum_{i=2m_2+1}^n x_{\pi(i)}
		- \frac{1}{m_2}\sum_{i=m_2+1}^{2m_2} x_{\pi(i)}
	}^p\\
	&=\frac{\xi_1^{p-1}}{m_2^p} (\mathrm{I}) + 
	\frac{\xi_2^{p-1}M^p}{m_1^p}(\mathrm{II}).
	\end{align*}
	
	To bound $(\mathrm{I})$, we apply the balanced weakly 
	dependent Khintchine inequality.  Let $I\subset\{1,\ldots,n\}$ with 
	cardinality $\abs{I}=M$. For such an index set $I$, let
	$\Pi_I = \{ \pi\in\mathbb{S}_n \,:\, \pi(\{2m_2+1,\ldots,n\})=I \}$.
	That is, $\pi\in\Pi_I$ maps the final $M$ indices into $I$.  
	Note that $\abs{\Pi_I} = {n\choose M}$.  As a result,
	\begin{align*}
	\xv_\pi\abs*{ 
		\sum_{i=1}^{m_2} x_{\pi(i)} - \sum_{i=m_2+1}^{2m_2} x_{\pi(i)}
	}^p 
	&\le
	\frac{1}{n!}\sum_{\pi\in\mathbb{S}_n}
	\abs*{ 
		\sum_{i=1}^{m_2} x_{\pi(i)} - \sum_{i=m_2+1}^{2m_2} x_{\pi(i)}
	}^p\\
	&\le
	\frac{(n-M)!}{n!}\sum_{\abs{I}=M} \frac{1}{(n-M)!}\sum_{\pi\in\Pi_I}
	\abs*{ 
		\sum_{i=1}^{m_2} x_{\pi(i)} - \sum_{i=m_2+1}^{2m_2} x_{\pi(i)}
	}^p\\
	&\le
	\frac{M!(n-M)!}{n!}\sum_{\abs{I}=M}\left[
	B_p(2m_2-1)^{p/2} (\sum_{i\notin I}x_i^2)^{p/2}
	\right]\\
	&\le
	B_p(2m_2-1)^{p/2} {s}_{n}^{p}.
	\end{align*}
	As the $x_i$ are centred, we have that 
	$(\sum_{i\notin I}x_i^2)^{p/2} \le s_n^p = (\sum x_i^2)^{p/2}$,
	and hence
	$$
	\frac{\xi_1^{p-1}}{m_2^p} \xv_\pi\abs*{ 
		\sum_{i=1}^{m_2} x_{\pi(i)} - \sum_{i=m_2+1}^{2m_2} x_{\pi(i)}
	}^p 
	\le \xi_1^{p-1} 
	B_p(2m_2^{-1})^{p/2}{s}_{n}^{p}.
	$$
	
	For $(\mathrm{II})$, 
	we first assume that $\kappa$ is a positive integer and
	$m_1 = \kappa m_2$ so $M = (\kappa-1)m_2$.
	In this case, we have
	$$
	\frac{\xi_2^{p-1}M^p}{m_1^p}(\mathrm{II}) = \xi_2^{p-1}\left(\frac{\kappa-1}{\kappa}\right)^{p}
	\xv_\pi\abs*{
		\sum_{i=m_2+1}^n \tilde{\delta}_ix_{i}
	}^p
	$$ 
	where $\tilde{\delta}_i$ are weighted Rademacher random variables
	with taking values $1/M$ or $-1/m_2$ such that $\sum \tilde{\delta}_i=0$.
	Applying Lemma~\ref{lem:convexity} gives
	$$
	\xv_\pi\abs*{
		\sum_{i=1}^n {\delta}_ix_{i}
	}^p \le
	\left\{
	B_p^{1/p}\left(\frac{2{s}_{n}^2}{m_2}\right)^{1/2} + 
	\left(\frac{\kappa-1}{\kappa}\right)\left(
	\xv_\pi\abs*{
		\sum_{i=m_2+1}^n \tilde{\delta}_ix_{i}
	}^p
	\right)^{1/p}
	\right\}^{p}.
	$$
	Noting that 
	$
	\xv_\pi\abs*{
		\sum_{i=m_2+1}^n \tilde{\delta}_ix_{i}
	}^p
	$
	is merely the original term to be bounded but with 
	$m_1 = \kappa m_2$ and $M = (\kappa-1)m_2$ replaced by 
	$(\kappa-1)m_2$ and $(\kappa-2)m_2$, respectively, 
	we apply this idea $\kappa-1$ more times to get 
	\begin{align*}
	\xv_\pi\abs*{
		\sum_{i=1}^n {\delta}_ix_{i}
	}^p 
	&\le
	B_p\left(\frac{2{s}_{n}^2}{m_2}\right)^{p/2}
	\left\{
	1 + 
	\left(\frac{\kappa-1}{\kappa}\right)\left(
	1 + \left(\frac{\kappa-2}{\kappa-1}\right)\left(
	\cdots (1 + 1/2)\cdots
	\right)
	\right)
	\right\}^{p}\\
	&\le
	B_p\left(\frac{2{s}_{n}^2}{m_2}\right)^{p/2}
	\left\{
	1 + \frac{\kappa-1}{\kappa} + 
	\frac{\kappa-2}{\kappa} + \ldots + \frac{1}{\kappa}
	\right\}^{p}\\
	&\le
	B_p\left(\frac{{s}_{n}^2}{m_2}\frac{(\kappa+1)^2}{2}\right)^{p/2}.
	\end{align*}
	Noting that $n = m_1+m_2=(\kappa+1)m_2$, we have
	$$
	\xv_\pi\abs*{
		\sum_{i=1}^n {\delta}_ix_{i}
	}^p 
	\le
	B_p{s}_{n}^p\left(\frac{(\kappa+1)^3}{2n}\right)^{p/2}.
	$$
	Now, consider $\kappa = a + r\in\rational$ with $a\in\natural$ 
	and
	$r\in[0,1)$.  Then,
	\begin{align*}
	B_p&\left(\frac{2{s}_{n}^2}{m_2}\right)^{p/2}
	\left\{
	1 + \frac{\kappa-1}{\kappa} + 
	\frac{\kappa-2}{\kappa} + \ldots + \frac{r}{\kappa}
	\right\}^{p}\\
	&\le 
	B_p\left(\frac{2{s}_{n}^2}{m_2}\right)^{p/2}
	\left\{
	\frac{a}{\kappa} + \frac{a-1}{\kappa} + 
	\frac{a-2}{\kappa} + \ldots + \frac{1}{\kappa} + \frac{(a+1)r}{\kappa}
	\right\}^{p}\\
	&\le
	B_p\left(\frac{2{s}_{n}^2}{m_2}\right)^{p/2}
	\frac{1}{\kappa^p}
	\left\{
	\frac{a(a+1)}{2} + (a+1)r
	\right\}^{p}\\
	&\le
	B_p\left(\frac{{s}_{n}^2}{m_2}\frac{(a+1)^2}{2}\right)^{p/2}
	\left\{
	\frac{a + 2r}{a+r}
	\right\}^{p}\\
	&=
	B_p\left(\frac{{s}_{n}^2}{m_2}\frac{(a+1)^2}{2}\right)^{p/2}
	\left\{
	1+\frac{r}{\kappa}
	\right\}^{p}.\\
	\end{align*}
	Noting further that $ra+r < a+r$ so that 
	$1 + r/(a+r) < 1+1/(a+1)$, we multiply by $(a+1)$ on each
	side to get $(a+1)(1+r/(a+r)) < a+2$.  Hence, 
	$$
	B_p\left(\frac{2{s}_{n}^2}{m_2}\right)^{p/2}
	\left\{
	1 + \frac{\kappa-1}{\kappa} + 
	\frac{\kappa-2}{\kappa} + \ldots + \frac{r}{\kappa}
	\right\}^{p}
	\le
	B_p\left(\frac{{s}_{n}^2}{m_2}\frac{(a+2)^2}{2}\right)^{p/2}.
	$$
	Hence, for $\kappa\in\natural$, we have 
	$\kappa+1 = \lceil\kappa+1\rceil$, and for $\kappa$ a non-integer
	we have 
	$
	a+2 = \lfloor\kappa\rfloor+2=\lceil\kappa\rceil+1=\lceil\kappa+1\rceil
	$. 
\end{proof}

\subsection{Kahane-Khintchine-type Inequalities}

Kahane extended Khintchine's inequality from the real
line to normed spaces \cite{KAHANE1964,LATALA1994}.   
The optimal value for the constant
$C_{p,p'}$ in Theorem~\ref{thm:Kahane} below 
is not known in the
case of interest for this article, $p>p'=2$; 
however, it has been conjectured
to be the same as in the real case, and as we see from the 
simulations and real data experiments, this conjecture seems 
to hold for our purposes.
In what follows, let 
$\mathcal{X}$ be a normed space with norm $\norm{\cdot}$.
Those spaces of statistical interest include $\real^d$, $L^2(0,1)$, and
spaces of matrices and positive definite trace class 
operators---i.e. covariance operators.

\begin{thm}[Kahane-Khintchine Inequality (1964)]
	\label{thm:Kahane}
	For any $p,p'\in[1,\infty)$,  there exists a universal finite 
	constant $C_{p,p'}>0$ such that for 
	any sequence of $X_1,\ldots,X_n\in\mathcal{X}$
	$$
	\left\{\xv\norm*{\sum_{i=1}^n \veps_i X_i}^p\right\}^{1/p}
	\le C_{p,p'} 
	\left\{\xv\norm*{\sum_{i=1}^n \veps_i X_i}^{p'}\right\}^{1/p'}
	$$
	where $\veps_i$ are iid Rademacher random variables.
\end{thm}

In general, we will consider the right hand side with $p'=2$, which
bounds the $p$th moments by the second moment.
For statistical applications, we are interested in a few specific 
setting for this theorem.  Namely, if $\mathcal{X}=\real^d$ for
$d\ge2$, then for the $\ell^q$ norm with $q\in[1,\infty]$, we have
$$
\left\{\xv\norm*{\sum_{i=1}^n \veps_i X_i}_{\ell^q}^p\right\}^{1/p}
\le C_{p} 
\norm*{ \left(\sum_{i=1}^n X_i\TT{X}_i\right)^{1/2} }_{S^q}
= C_{p} 
(n-1)^{1/2}\norm*{ \hat{\Sigma}^{1/2} }_{S^q}
$$
where $\norm{\cdot}_{S^q}$ is the $q$-Schatten norm and $\hat{\Sigma}$
is the empirical covariance estimator for the $X_i$.
Similarly, in the functional data setting, 
if $X_i$ are continuous and in $L^q[0,1]$, then the right hand
side becomes $C_p(n-1)^{1/2}\norm{\hat{\Sigma}(s,s')^{1/2}}_{S^q}$
where $\hat{\Sigma}:[0,1]^2\rightarrow\real$ is the empirical 
covariance operator.

For non-commutative Banach spaces \citep{PISIERXU2003}, such as when $X_i$ are real 
valued matrices, we have a slightly different bound.  Let 
$\mathcal{X}=\real^{d\times d'}$.  Then, with respect to the 
$q$-Schatten norm,
$$
\left\{\xv\norm*{\sum_{i=1}^n \veps_i X_i}_{S^q}^p\right\}^{1/p}
\le C_{p} 
\max\left\{
\norm*{ \left(\sum_{i=1}^n X_i\TT{X}_i\right)^{1/2} }_{S^q},
\norm*{ \left(\sum_{i=1}^n \TT{X}_i{X_i}\right)^{1/2} }_{S^q}
\right\}.
$$

The above results all have iid $\veps_i$.  Applying
similar methods as in \cite{SPEKTOR2016} and as in the 
previous section, we can consider the moment bounds under 
weak dependency conditions on the $\veps_i$.
This theorem is stated for balanced samples with
adjustments for imbalanced samples omitted as they 
follow exactly as in the previously discussed 
real valued setting.

\begin{thm}[Kahane-Khintchine with Weak Dependence]
	\label{thm:kahaneWDep}
	Let $\veps_i$ are Rademacher random variables such that
	$\sum \veps_i=0$.  Furthermore, let $p\in[1,\infty)$.
	
	For commutative Banach spaces there exists a 
	universal finite constant $C_{p}>0$ such that for 
	any sequence of $X_1,\ldots,X_n\in\mathcal{X}$
	$$
	\left\{\xv\norm*{\sum_{i=1}^n \veps_i X_i}^p\right\}^{1/p}
	\le C_{p}2^{1/2} 
	\norm*{\left(\sum_{i=1}^n X_iX_i^\star\right)^{1/2}}.
	$$
	
	For non-commutative Banach spaces 
	there exists a universal finite 
	constant $C_{p}>0$ such that for 
	any sequence of $X_1,\ldots,X_n\in\mathcal{X}$
	$$
	\left\{\xv\norm*{\sum_{i=1}^n \veps_i X_i}^p\right\}^{1/p}
	\le C_{p}2^{1/2}\max\left\{
	\norm*{\left(\sum_{i=1}^n X_iX_i^\star\right)^{1/2}},
	\norm*{\left(\sum_{i=1}^n X_i^\star X_i\right)^{1/2}}
	\right\} .
	$$
\end{thm}

Before proving this theorem, we discuss some preliminary 
results regarding Schatten norms.  Let  
$\preceq$ denote positive semi-definite ordering.
For positive semi-definite
$q$-Schatten class linear operators $\Gamma$ and $\Delta$
with $0\preceq\Gamma\preceq\Delta$, 
$$
\norm{\Gamma}_{S^q} \le \norm{ \Delta }_{S^q},
\text{ and } 
\norm{ (\Gamma\Gamma^\star)^{1/2} }_{S^q} = 
\norm{ \Gamma\Gamma^\star }_{S^{q/2}}^{1/2}
$$
where the square root is well defined as $\Gamma\Gamma^\star$
is symmetric positive semi-definite.  Lastly, 
via direct calculation,
\begin{align*}
(\Gamma-\Delta)(\Gamma-\Delta)^\star &\preceq
2(\Gamma\Gamma^\star - \Delta\Delta^\star) \\
(\Gamma-\Delta)^\star(\Gamma-\Delta) &\preceq
2(\Gamma^\star\Gamma - \Delta^\star\Delta).
\end{align*}

\begin{proof}
	For $\mathbb{S}_n$ the symmetric group on $n$ elements, let 
	$f:\mathbb{S}_n\rightarrow\real$ by
	$$
	f(\pi) := \norm*{ 
		\sum_{i=1}^m X_{\pi(i)} - \sum_{i=m+1}^{2m}X_{\pi(i)} 
	}
	$$
	For $k = 1,\ldots,m$, we define $B_{k,\pi} = X_{\pi(k)}-X_{\pi(k+m)}$
	and $H_{k,\pi} = \sum_{i=k+1}^{m} B_{i,\pi} = 
	\sum_{i=k+1}^mX_{\pi(i)} - \sum_{i=m+k+1}^{2m}X_{\pi(i)}$ where
	$H_{m,\pi}=0$ being an empty sum.
	
	Note that the $B_{k,\pi}$ are symmetric random variables for 
	$\pi$ uniform on $\mathbb{S}_n$.  Thus, 
	$\xv_\pi\norm{f(\pi)} = \xv_\pi\norm{B_{1,\pi}+H_{1,\pi}} =
	\xv_\pi\norm{-B_{1,\pi}+H_{1,\pi}}
	$ and furthermore, letting $\delta_1,\ldots,\delta_m$ be iid
	Rademacher random variables,
	\begin{align*}
	\xv_\pi\norm{f(\pi)}^p 
	&= \xv_\pi\xv_{\delta_1}\norm{ \delta_1B_{1,\pi} + H_{1,\pi} }\\
	&= \xv_\pi\xv_{\delta_1}\xv_{\delta_2}\norm{ 
		\delta_1B_{1,\pi} + \delta_2B_{2,\pi} + H_{2,\pi} 
	}\\
	&= \xv_\pi\xv_{\delta_1}\ldots\xv_{\delta_m}\norm*{ 
		\sum_{i=1}^m \delta_iB_{i,\pi}  
	}
	\end{align*}
	From here, we consider separately the commutative and 
	non-commutative settings.
	
	For the commutative setting, we apply the facts about Schatten
	norms preceding this proof.
	Beginning with the classic Kahane-Khintchine inequality from above
	with $p'=2$, we have 
	$$
	\xv_\pi\norm{f(\pi)}_q^p \le
	C_p
	\norm*{ \left(\sum_{i=1}^m B_{i,\pi}{B_{i,\pi}^\star}\right)^{1/2}}_{S^q}^{p}.
	$$
	Noting that 
	$
	B_{i,\pi}{B_{i,\pi}^\star} \le 2( X_{\pi(k)}{X}_{\pi(k)}^\star + 
	X_{\pi(k+m)}{X}_{\pi(k+m)}^\star ),
	$
	\begin{multline*}
	\norm*{ \left(\sum_{i=1}^m B_{i,\pi}{B_{i,\pi}^\star}\right)^{1/2}}_{S^q}^{p}
	= 
	\norm*{ \sum_{i=1}^m B_{i,\pi}{B_{i,\pi}^\star}}_{S^{q/2}}^{p/2}\\
	\le 
	2^{p/2}\norm*{ \sum_{i=1}^{2m} X_{i}{X_{i}^\star}}_{S^{q/2}}^{p/2}
	=
	2^{p/2}\norm*{ \left(\sum_{i=1}^{2m} X_{i}{X_{i}^\star}\right)^{1/2}}_{S^{q}}^{p}
	\end{multline*}
	
	For the non-commutative setting, we proceed as 
	before using the non-commutative
	variant of Kahane-Khintchine and also noting that
	$
	B_{i,\pi}^\star{B_{i,\pi}} \le 2( X_{\pi(k)}^\star{X}_{\pi(k)} + 
	X_{\pi(k+m)}^\star{X}_{\pi(k+m)} ).
	$
	\begin{align*}
	\xv_\pi\norm{f(\pi)}_q^p 
	&\le
	C_p
	\max\left\{
	\norm*{ 
		\left(\sum_{i=1}^m B_{i,\pi} B_{i,\pi}^\star\right)^{1/2} 
	}_{S^q}^p,
	\norm*{ 
		\left(\sum_{i=1}^m B_{i,\pi}^\star B_{i,\pi}\right)^{1/2} 
	}_{S^q}^p
	\right\}\\
	&=
	C_p
	\max\left\{
	\norm*{ 
		\sum_{i=1}^m B_{i,\pi} B_{i,\pi}^\star
	}_{S^q}^{p/2},
	\norm*{ 
		\sum_{i=1}^m B_{i,\pi}^\star B_{i,\pi}
	}_{S^q}^{p/2}
	\right\}\\
	&\le
	C_p2^{p/2}
	\max\left\{
	\norm*{ 
		\sum_{i=1}^{2m} X_{i,\pi} X_{i,\pi}^\star 
	}_{S^q}^{p/2},
	\norm*{ 
		\sum_{i=1}^{2m} X_{i,\pi}^\star X_{i,\pi}
	}_{S^q}^{p/2}
	\right\}\\
	&=
	C_p2^{p/2}
	\max\left\{
	\norm*{ 
		\left(\sum_{i=1}^{2m} X_{i,\pi} X_{i,\pi}^\star\right)^{1/2} 
	}_{S^q}^{p},
	\norm*{ 
		\left(\sum_{i=1}^{2m} X_{i,\pi}^\star X_{i,\pi}\right)^{1/2}
	}_{S^q}^{p}
	\right\}
	\end{align*}
\end{proof}

\subsubsection{On Optimal Constants}

For the classic Khintchine inequality, the optimal constants
due to \cite{HAAGERUP1981} coincide with the lower bound 
imposed by the central limit theorem.  That is, 
Khintchine's inequality states that 
$
\xv\abs*{ \sum x_i\veps_i }^p \le 
B_p\norm*{x}_2^{p} 
$
where 
$$
B_p = 2^{p/2}\Gamma\left\{(p+1)/2\right\}/\sqrt{\pi}.
$$
This coincides precisely with the $p$th absolute 
moment of a standard normal random 
variable---i.e. $\xv \abs{Z}^p = B_p$ for $Z\dist\distNormal{0}{1}$.

For the Kahane-Khintchine inequality, optimal constants are 
not currently known.\footnote{For the lower bound, optimal 
	constants are known due to \cite{LATALA1994}.}  
However, it is strongly conjectured 
that they coincide with those in the standard 
Khintchine inequality.
Moreover in the multivariate setting, 
due again to the central limit theorem, the 
optimal constant has a lower bound.  Indeed,
let $Z\dist\distNormal{0}{\Sigma}$, then
$$
{\xv\norm{Z}_{\ell^q}^q}\norm{\Sigma^{1/2}}_{S^q}^{-1}
=2^{p/2}\Gamma\left\{(p+1)/2\right\}/\sqrt{\pi}.
$$
This can be extended into a functional data setting
using the fact that the space of covariance operators 
arises from the closure of the set of finite rank operators---i.e.
the multivariate setting.

\subsection{Sub-Gaussian Concentration}

Given upper bounded on the $p$th moments of a random 
permutation statistic, we want to quantify the concentration
behaviour.  In particular, we want as sharp an upper bound 
as possible to achieve the best statistical power for 
hypothesis testing.

We first consider the standard moment bounds to achieve 
sub-Gaussian concentration \citep{BOUCHERON2013} in 
Proposition~\ref{prop:subGauss}.
This is improved if $X$ is symmetric \citep{GARLING2007}
in Proposition~\ref{prop:subGaussSym}.
Lastly, even if the moment condition is weakened as
in Proposition~\ref{prop:subGaussSymWeak}, we still have
sub-Gaussian concentration.
\begin{prop}
	\label{prop:subGauss}
	For a centred univariate random variable $X\in\real$ such that  
	$
	\xv{\abs{X}}^{2p} \le p!C^{p}
	$
	for some constant $C>0$.
	Then, 
	$$
	\prob{ X>t } \le \ee^{-t^2/8C}.
	$$
\end{prop}
\begin{prop}
	\label{prop:subGaussSym}
	For a centred symmetric univariate random variable 
	$X\in\real$ such that  
	$
	\xv{\abs{X}}^{2p} \le p!C^{p}
	$
	for some constant $C>0$.
	Then, 
	$$
	\prob{ X>t } \le \ee^{-t^2/2C}.
	$$
\end{prop}

\begin{remark}
	Note that the difference between the above two propositions is a factor of 
	4 in the denominator of the exponent.  This stems from a standard 
	symmetrization trick where one considers $X$ and $X'$, an iid copy of $X$,
	so that 
	$$
	\xv\abs{ X - X' }^{2p} \le 2^{2p} \xv\abs{X}^{2p} \le (4C)^{p}p!.
	$$
	Thus, the following results can be similarly adjusted for asymmetric 
	random variables.
\end{remark}

\begin{prop}
	\label{prop:subGaussSymWeak}
	For a centred symmetric univariate random variable 
	$X\in\real$ such that  
	$
	\xv{\abs{X}}^{2p} \le (2p)!C^{p}/p!
	$
	for some constant $C>0$.
	Then, 
	$$
	\prob{ X>t } \le \ee^{-t^2/4C}.
	$$
\end{prop}
\begin{proof}
	The moment generating function is 
	$$
	\xv \ee^{\lmb Z} 
	=   \sum_{p=0}^\infty \frac{\lmb^p\xv Z^p}{p!}
	=   \sum_{p=0}^\infty \frac{\lmb^{2p}\xv Z^{2p}}{(2p)!}
	\le \sum_{p=0}^\infty \frac{\lmb^{2p}C^p}{p!}
	\le \ee^{ \lmb^2C }.
	$$
	The result follows from Markov's (Chernoff's) Inequality.
\end{proof}

\section{Proofs of main theorems}
\label{app:mainProofs}
Now that all of the results from the previous section 
have been established, we prove the tail bounds on the
test statistics of interest by (1) applying the 
appropriate Khintchine-type moment bound and (2)
applying the appropriate sub-Gaussian bound on the
moment generating function.

\begin{proof}[Proof of Theorem~2.1] 
	For the balanced case of $\kappa=1$,
	let $n=2m$ and $\veps_1,\ldots,\veps_n$ be Rademacher random variables
	such that $\sum_{i=1}^n \veps_i = 0$---i.e. not independent.  
	Then, we can 
	rewrite $T(\pi)$ from equation~2.1 as 
	$$
	T(\pi) = \frac{1}{s{m}}\sum_{i=1}^n \veps_i X_i.
	$$
	Treating $X_i\in\real$ as fixed, we can use
	Theorem~\ref{thm:conditionalKhintchine}
	to bound the $p$th
	absolute moment of $T(\pi)$ for $\pi$ uniformly distributed
	on $\mathbb{S}_n$,
	$$
	\xv_\veps\abs*{T(\pi)}^p =
	\left(\frac{1}{s{m}}\right)^{p}\xv_\veps\abs*{\sum \veps_i X_i}^p
	\le B_p\left(\frac{ 
		\norm{X}_2^2-n\bar{X}^2}{
		s^2m^2
	}
	\right)^{p/2}.
	$$
	However, the term $\norm{X}_2^2-n\bar{X}^2 = (n-1){s}^2 < 2ms^2$.
	Hence, the result of \cite{SPEKTOR2016} can be equivalently 
	rewritten as 
	$$
	\xv\abs*{
		T(\pi)
	}^{2p}
	\le (2/m)^pB_{2p} = \frac{2^p(2p)!}{m^pp!}.
	$$
	Applying Proposition~\ref{prop:subGaussSymWeak} gives the
	desired result.
	
	For $\kappa>1$---i.e. the imbalanced setting---we
	apply Theorem~\ref{thm:unbalancedKhintchine} to get moment bounds
	$$
	\xv\abs*{
		T(\pi)
	}^{2p}
	\le \left(\frac{(\kappa+1)^2}{2m_2}\right)^p\frac{(2p)!}{p!}.
	$$
	and 
	Proposition~\ref{prop:subGaussSymWeak} again to get the 
	desired result.
\end{proof}

\begin{proof}[Proof of Theorem~2.2] 
	As with the previous proof, let
	$n=2m$ and $\veps_1,\ldots,\veps_n$ be Rademacher random variables
	such that $\sum_{i=1}^n \veps_i = 0$.
	Our permuted test statistic is 
	$T(\pi) = \norm*{\sum_{i=1}^n \veps_i X_i}_q$.
	We apply Theorem~\ref{thm:kahaneWDep}, our 
	Kahane-Khintchine variant assuming the above dependency on 
	$\veps$, in the commutative Banach setting to get 
	$$
	\xv T(\pi)^p \le
	C_p^p 2^{p/2} \norm*{ \left(\sum_{i=1}^n X_iX_i^*\right)^{1/2} }^p.
	$$
	Note that while the optimal constant is not known,
	$C_p \sim p^{1/2}$ from the central limit theorem and 
	from the proof in \cite{DIESTEL1995}, Chapter 11.
	Hence, applying the fact that $((2p)!/p!)^{1/2p}\sim p^{1/2}$ and
	Proposition~\ref{prop:subGaussSymWeak}.  We have the 
	desired result.
\end{proof}

\begin{proof}[Proof of Theorem~2.3] 
	This proof is identical to that for Theorem~2.2 
	except we apply the non-commutative variant of Kahane-Khintchine.
\end{proof}

\begin{proof}[Proof of Proposition~2.5] 
	We note first that $nT_0^2/(2+\kappa+\kappa^{-1})$ is 
	approximately $\distChiSquared{1}$ via the central
	limit theorem.  Hence, for $Z\dist\distChiSquared{1}$,
	some $c>0$,
	and some $u\in(0,1)$,
	\begin{align*}
	\prob{ \ee^{-Z/c} \le u }
	&= \prob{ Z \ge -c\log u }\\
	&= (2\pi)^{-1/2}\int_{-c\log u}^\infty 
	x^{-1/2}\ee^{-x/2}dx\\
	&= \left(\frac{c}{2\pi}\right)^{1/2}
	\int_{0}^u (-\log y)^{-1/2}y^{c/2-1}dy\\
	&\le\left(\frac{c}{2\pi}\right)^{1/2}
	\int_{0}^u (1 - y)^{1/2-1}y^{c/2-1}dy\\
	&= \frac{(c/2)^{1/2}\Gamma(c/2)}{\Gamma((c+1)/2)}
	I(u;c/2,1/2)
	\end{align*}
	where we use the inequality $-\log y \ge 1-y$ for 
	$y\in(0,1)$.
	The coefficient
	${(c/2)^{1/2}\Gamma(c/2)}{\Gamma((c+1)/2)}^{-1}\rightarrow1$
	as $c\rightarrow\infty$.  Replacing 
	$c$ with $2\lceil\kappa+1\rceil^3/(2+\kappa+\kappa^{-1})$, we 
	conclude that
	$$
	\prob{
		\exp\left\{
		-nT(\pi)^2/2\lceil\kappa+1\rceil^3
		\right\} < u
	} \le C_0I\left(
	u ; 
	\frac{\lceil\kappa+1\rceil^3}{(2+\kappa+\kappa^{-1})},
	\frac{1}{2}
	\right)
	$$
	where 
	$
	C_0 = {\left(\frac{
			\lceil\kappa+1\rceil^{3}
		}{
			2+\kappa+\kappa^{-1}
		}\right)^{1/2}
		\Gamma\left(
		\frac{\lceil\kappa+1\rceil^3}{2+\kappa+\kappa^{-1}}
		\right)}{
		\Gamma\left(
		\frac{1}{2}+\frac{\lceil\kappa+1\rceil^3}{2+\kappa+\kappa^{-1}}
		\right)^{-1}
	}
	$.
\end{proof}

\begin{proof}[Proof of Theorem~2.4] 
	Let $Z = h(\norm{X})\in\real^+$, and let 
	$B^*$ be a countable dense subset of the unit ball 
	of the dual space $\mathcal{X}^*$, which consists 
	of bounded linear functionals $\phi$.  
	Then, we can write 
	$Z = \sup_{\phi\in B^*}h(\phi(X))$ being a countable
	supremum.  Via application of 
	Talagrand's concentration inequality
	\citep{TALAGRAND1996CON},
	we have that 
	$$
	\prob{
		Z \ge \xv Z + t
	} \le \exp\left( \frac{-t^2}{a+bt} \right)
	$$
	for positive constants $a$ and $b$ depending on 
	$\xv h(\norm{X})^2$ and 
	$\sup_{X\in\mathcal{X}}h(\norm{X})$.\footnote{
		Refined values for such constants can be found in 
		other works \citep{BOUSQUET2003,KLEINRIO2005,GINENICKL2015},
		but are not pertinent to this discussion.
	}
	Noting that for $t\ge0$
	$$
	\frac{d}{dt}\left\{\frac{t}{1+bt/a}\right\}
	= \frac{1}{(1+bt/a)^2} 
	\le \frac{1}{1+bt/a} =  
	\frac{d}{dt}\left\{\frac{a}{b}\log(1+bt/a)\right\},
	$$
	we have that 
	\begin{multline*}
	\exp\left( \frac{-t^2}{a+bt} \right)
	= \exp\left\{
	-\frac{1}{b}\left(
	t - \frac{at}{a+bt}
	\right)
	\right\} 
	\\
	\le
	\exp\left\{
	-\frac{1}{b}\left(
	t - \frac{a}{b}\log(1+bt/a)
	\right)
	\right\}
	= \ee^{-t/b}\left(
	1 + \frac{b}{a}t
	\right)^{a/b^2}.
	\end{multline*}
	If $a\notin\natural$, then we 
	replace $a$ with $\lceil a\rceil$. 
	Then, we have that 
	\begin{multline*}
	\exp\left( \frac{-t^2}{a+bt} \right)
	\le \ee^{-t/b}\left[\left(
	1 + \frac{bt}{a}
	\right)^{a}\right]^{1/b^2}\\
	= \ee^{-t/b}\left[
	\sum_{k=0}^{a}  {a\choose k}\left(
	\frac{bt}{a}
	\right)^k
	\right]^{1/b^2} 
	\le \ee^{-t/b}\left[
	\sum_{k=0}^{a}  \frac{1}{k!}\left(
	bt
	\right)^k
	\right]^{1/b^2}.
	\end{multline*}
	If $b=1$, then this is just the distribution function
	of the Erlang (gamma) distribution with 
	shape parameter $a$ and scale parameter $1$.
	More generally,
	we have that 
	$$
	\left[
	\sum_{k=0}^{a}  \frac{1}{k!}\left(bt\right)^k
	\right]^{1/b^2}
	= \ee^{t/b}\left[1 - \ee^{-bt}
	\sum_{k=a+1}^{\infty}  \frac{1}{k!}\left(
	bt
	\right)^k
	\right]^{1/b^2}
	$$
	whose $l$th derivative for $b>1$, denoting the Pochhamer symbol
	$(b^{-2})_j=\prod_{i=1}^j(b^{-2}-i+1)$, 
	can be written as 
	\begin{align*}
	\frac{d^l}{dt^l}&
	\left[
	\sum_{k=0}^{a}  \frac{1}{k!}\left(bt\right)^k
	\right]^{1/b^2}\\
	&=
	\frac{\ee^{t/b}}{b^l}
	\left[1 - \ee^{-bt}
	\sum_{k=a+1}^{\infty}  \frac{1}{k!}\left(
	bt
	\right)^k
	\right]^{1/b^2}\\ 
	&~~~- \ee^{t/b}\sum_{j=1}^l b^{j-l}(b^{-2})_j
	\left[1 - \ee^{-bt}
	\sum_{k=a+1}^{\infty}  \frac{1}{k!}\left(
	bt
	\right)^k
	\right]^{1/b^2-j}
	\frac{d^{j}}{dt^{j}} \left\{\ee^{-bt}
	\sum_{k=a+1}^{\infty}  \frac{1}{k!}\left(
	bt
	\right)^k\right\}\\
	&=
	\frac{\ee^{t/b}}{b^l}
	\left[1 - \ee^{-bt}
	\sum_{k=a+1}^{\infty}  \frac{1}{k!}\left(bt\right)^k
	\right]^{1/b^2} \\
	&~~~- \ee^{t/b}\sum_{j=1}^l b^{2j-l}(b^{-2})_j
	\left[1 - \ee^{-bt}
	\sum_{k=a+1}^{\infty}  \frac{1}{k!}\left(
	bt
	\right)^k
	\right]^{1/b^2-j}
	\ee^{-bt}
	\sum_{i=0}^j 
	(-1)^{j-i} 
	\sum_{k=0\vee (a+1-i)}^{\infty}  \frac{1}{k!}\left(bt\right)^k.
	\end{align*}
	Thus, for $l\le a$, we have that 
	$$
	\left.\frac{d^l}{dt^l}
	\left[
	\sum_{k=0}^{a}  \frac{1}{k!}\left(bt\right)^k
	\right]^{1/b^2}\right|_{t=0} = \frac{1}{b^l} 
	$$
	as the second term vanishes,
	and for $l\ge a+1$ and $b>1$, we have that 
	\begin{align*}
	\left.\frac{d^l}{dt^l}
	\left[
	\sum_{k=0}^{a}  \frac{1}{k!}\left(bt\right)^k
	\right]^{1/b^2}\right|_{t=0} &= 
	\frac{1}{b^l}\left[1
	- \sum_{j=a+1}^l b^{2j}(b^{-2})_j
	\sum_{i=a+1}^j 
	(-1)^{j-i} 
	\right]\\
	&= \frac{1}{b^l}\left[1
	- \sum_{\substack{j=a+1\\j=a+1 \text{ mod } 2}}^l 
	\prod_{i=1}^{j} (1 - (i+1)b^2)
	\right].
	\end{align*}
	which is negative for odd $a$.  Thus, for $a$ odd---in
	the case where $a\in\real^+$, we replace $a$ with 
	$2\lfloor a/2 \rfloor+1$---we can finally bound via 
	$a$th order approximation
	$$
	\exp\left( \frac{-t^2}{a+bt} \right)
	\le \ee^{-t/b}\left[\left(
	1 + \frac{bt}{a}
	\right)^{a}\right]^{1/b^2}
	\le \ee^{-t/b}
	\sum_{k=0}^{a}  \frac{1}{k!}\left(
	\frac{t}{b}
	\right)^k
	= \int_{t}^\infty
	\frac{x^{a-1}\ee^{-x/b}}{b^a\Gamma(a)}dx
	$$
	being once again the Erlang (gamma) distribution function 
	with shape parameter $a$ and scale parameter $b$.
	
	As a result, we have for $u\in(0,1)$, 
	$C$ some positive constant, and $I(u;a,c/b-a)$
	the incomplete beta function where 
	$c$ is
	chosen large enough so that $c/b-a>0$,
	\begin{align*}
	\prob{ \ee^{-(Z-\xv{Z})/c} \le u  }
	&= \prob{ Z-\xv{Z}\ge -c \log u }\\
	&\le 
	\int_{-c\log u}^\infty 
	\frac{x^{a-1}\ee^{-x/b}}{b^a\Gamma(a)}dx\\
	&=
	\frac{c^a}{b^a\Gamma(a)}\int_{0}^u
	(-\log y)^{a-1}{y^{c/b-1}}dy\\
	&\le
	\frac{c^a}{b^a\Gamma(a)}\int_{0}^u
	(1 - y)^{a-1}{y^{c/b-a-1}}dy = C I(u;a,c/b-a)
	\end{align*}
	where the final inequality comes from 
	$-\log y = \sum_{k=1}^\infty (1-y)^k/k \le (1-y)/y$ 
	for $0<y<1$.
\end{proof}

\begin{proof}[Proof of Theorem~3.1] 
	Let $T = \norm{\TT{X}\veps}_{\ell^2}$ and note
	for $n=2m$ that
	$T^2 = \sum_{i=1}^n a_{i,j}\veps_{i}\veps_{j}$
	where $a_{i,j}$ is the $ij$th entry of $X\TT{X}$.
	This is an homogeneous Rademacher chaos of order 2.
	
	As in \cite{SPEKTOR2016}, we note the following 
	correspondence.  Let 
	$
	\Omega = \{
	\veps \in \{\pm1\}^n \,|\, \sum\veps_i=0
	\},
	$
	then 
	$$
	\pi \in \mathbb{S}_n \longleftrightarrow
	\left\{
	\veps\in\Omega\,|\,
	\veps_i=1 \text{ if }\pi(i)\le m\text{ and }
	\veps_i=-1\text{ if }\pi(i)> m
	\right\}.
	$$
	Hence, for any $\pi\in\mathbb{S}_n$, we can write
	$$
	T^2(\pi) = 
	\sum_{i\le m,j\le m} a_{\pi(i),\pi(j)} -
	\sum_{i  > m,j\le m} a_{\pi(i),\pi(j)} -
	\sum_{i\le m,j  > m} a_{\pi(i),\pi(j)} +
	\sum_{i  > m,j  > m} a_{\pi(i),\pi(j)}
	$$
	and consider 
	$$
	\xv_\veps \abs{T^2}^p = 
	\xv_\pi \abs{T^2(\pi)}^p = 
	\xv_\pi \abs*{
		\sum_{i=1}^m\left\{
		a_{\pi(i),\pi(j)} - a_{\pi(i+m),\pi(j)} -
		a_{\pi(i),\pi(j+m)} + a_{\pi(i+m),\pi(j+m)} 
		\right\}
	}^p.
	$$
	Writing 
	$
	b_{k,k',\pi} = a_{\pi(k),\pi(k')} - a_{\pi(k+m),\pi(k')} -
	a_{\pi(k),\pi(k'+m)} + a_{\pi(k+m),\pi(k'+m)},
	$
	and $H_{k,\pi} = \sum_{i,j\in I_k}b_{i,j,\pi}$
	where the sum is over $I_{k} = \{
	1\le i,j\le m \,|\, i+j>k+1
	\}$
	with the empty sum being zero, 
	we note that 
	$$
	T^2(\pi) = b_{1,1,\pi} + H_{1,\pi} = 
	b_{1,1,\pi} + b_{1,2,\pi} + b_{2,1,\pi} + b_{2,2,\pi} + H_{2,\pi}
	= \ldots = \sum_{i,j=1}^m b_{i,j,\pi}.
	$$
	Then,
	\begin{align*}
	\xv_\pi\abs{T^2(\pi)}^p 
	&= \xv_\pi\abs{
		b_{1,1,\pi} + b_{1,2,\pi} + b_{2,1,\pi} + b_{2,2,\pi} + H_{2,\pi}
	}^p\\
	&= \xv_\pi\abs{
		b_{1,1,\pi} - b_{1,2,\pi} - b_{2,1,\pi} + b_{2,2,\pi} + H_{2,\pi}
	}^p\\
	&= \xv_\pi\xv_\delta\abs{
		\delta_1\delta_1b_{1,1,\pi} + 
		\delta_1\delta_2b_{1,2,\pi} - 
		\delta_2\delta_1b_{2,1,\pi} + 
		\delta_2\delta_2b_{2,2,\pi} + H_{2,\pi}
	}^p
	\end{align*}
	for $\delta_1,\delta_2$ iid Rademacher random variables.
	Continuing in this fashion, we have 
	$
	\xv_\pi\abs{T^2(\pi)}^p = 
	\xv_\pi\xv_\delta\abs{ 
		\sum_{i,j=1}^m \delta_i\delta_jb_{i,j,\pi}
	}^p.
	$ 
	From here we apply Corollary~3 from \cite{KWAPIEN1987}.
	
	First note that as the $\delta_i$ are just iid Rademacher
	random variables, the standard Khintchine (or Kahane-Khintchine) 
	inequality applies with coefficient $B_{2p}=((2p)!/2^pp!)^{1/2p}$.
	Then Corollary~3 from \cite{KWAPIEN1987} to this degree 2 
	polynomial chaos implies that
	$$
	\left[\xv_\pi\xv_\delta\abs*{  
		\sum_{i,j=1}^m \delta_i\delta_jb_{i,j,\pi}
	}^p\right]^{1/p}
	\le B_p^2 C 
	\left[\xv_\pi\xv_\delta\abs*{  
		\sum_{i,j=1}^m \delta_i\delta_jb_{i,j,\pi}
	}^2\right]^{1/2}
	$$
	where $C$ is a universal constant which for homogeneous degree $d$
	polynomial chaoses is $d^{3d}/d!$ or simply $2^5=32$ in our case.
	The expectation on the right hand side then becomes
	$$
	\left[\xv_\pi\xv_\delta\abs*{  
		\sum_{i,j=1}^m \delta_i\delta_jb_{i,j,\pi}
	}^2\right]^{1/2} 
	= 
	\left[\xv_\pi  
	\sum_{i,j=1}^m b_{i,j,\pi}^2
	\right]^{1/2} 
	\le
	\left[  
	\sum_{i,j=1}^m 4a_{i,j}^2
	\right]^{1/2}
	=  
	2\norm{X\TT{X}}_{S^2}.
	$$
	
	Absorbing the $2$ into $C$, we have the moment bounds
	$$
	\xv_\pi\abs{T(\pi)}^{2p} \le
	B_p^{2p}C^p\norm{X\TT{X}}_{S^2}^p.
	$$
	To adapt these moment bounds into a tail bound, we use the 
	standard moment generating function approach, but in 
	preparation we first recall the Legendre duplication formula  
	$\Gamma(2p)=2^{2p-1}\Gamma(p)\Gamma(p+1/2)/\sqrt{\pi}$ 
	and then note the following:
	\begin{multline*}
	\frac{2^{p}(\Gamma(p+1))^2}{\Gamma(2p+1)(\Gamma(p/2+1))^2} = 
	\frac{2^{p}p^2(\Gamma(p))^2}{2p\Gamma(2p)(p/2)^2(\Gamma(p/2))^2} =
	\frac{2^{p+1}(\Gamma(p))^2}{p\Gamma(2p)(\Gamma(p/2))^2} = \\
	\frac{2^{-p+2}\sqrt{\pi}\Gamma(p)}{
		p\Gamma(p+1/2)(\Gamma(p/2))^2
	} \le
	\frac{\sqrt{\pi}\Gamma(p)}{2^{p-2}p\Gamma(p+1/2)}
	\frac{p(2\ee)^p}{4\pi p^p} =
	\frac{\Gamma(p)}{\Gamma(p+1/2)} 
	\frac{\ee^p}{\sqrt{\pi}p^p} \le\\ 
	\frac{\sqrt{p+\pi^{-1}}\ee^p}{\sqrt{\pi}p^{p+1}} \le
	\left[\frac{1}{\sqrt{\pi}}+\frac{1}{\pi\sqrt{p}}\right]
	\frac{\ee^p}{p^{p+1/2}} \le
	\left[\frac{1}{\sqrt{\pi}}+\frac{1}{\pi\sqrt{p}}\right]
	\frac{\ee}{p!} \le 
	\left[\frac{\ee}{\sqrt{\pi}}+\frac{\ee}{\pi}\right]\frac{1}{p!},
	\end{multline*}  
	because, via Watson's formula \citep{WATSON1959},
	$$
	\frac{\Gamma(p)}{\Gamma(p+1/2)} =
	\frac{\Gamma(p+1)}{p\Gamma(p+1/2)}
	\le \frac{\sqrt{p+\pi^{-1}}}{p}.
	$$
	Let $\pi$
	and $\pi'$ be independent uniform random permutations from
	$\mathbb{S}_n$.  Then, updating $C$ as necessary, 
	\begin{align*}
	\xv_\pi\exp( \lmb T(\pi) )
	&\le
	\xv_\pi\exp( \lmb (T(\pi)-T(\pi') )\\
	&\le
	\sum_{p=1}^\infty
	\frac{\lmb^p}{p!}\xv_\pi\abs{T(\pi)-T(\pi')}^p\\
	&\le
	\sum_{p=1}^\infty
	\frac{\lmb^{2p}2^{2p}C^p}{(2p)!}
	\frac{(p!)^2}{2^p((p/2)!)^2}
	\norm{X\TT{X}}_{S^2}^p\\
	&\le
	\sum_{p=1}^\infty
	\frac{\lmb^{2p}C^p}{p!}\norm{X\TT{X}}_{S^2}^p\\
	&\le \ee^{\lmb^2C\norm{X\TT{X}}_{S^2}}, 
	\end{align*}
	which gives the desired sub-Gaussian concentration as in 
	Proposition~\ref{prop:subGaussSymWeak}.
\end{proof}

\section{Additional Data Experiments}
\label{app:additionalData}

\subsection{Multivariate Data}
We test the performance of the bound in Theorem~2.2 
on simulated multivariate Gaussian data in 
$\ell^q(\real^{12})$ for $q=1,2,\infty$.
The sample size is $m_1=m_2=50$.  
Figure~\ref{fig:multVTest} displays the result of running such a two-sample
test for each of the three norms compared to the standard 
permutation test approximated by sampling 1000 permutations.
This was replicated 1000 times and the average $\log_2$ p-values
are plotted.  We see that the Kahane bound does not achieve as 
much power as the standard permutation test.  However, after applying
the empirical beta adjustment from Section~2.4 
with moments computed via 10 permutations, the computed p-values
align perfectly with the standard permutation test.

\begin{figure}
	\begin{center}
		\includegraphics[width=0.4\textwidth]{\PICDIR/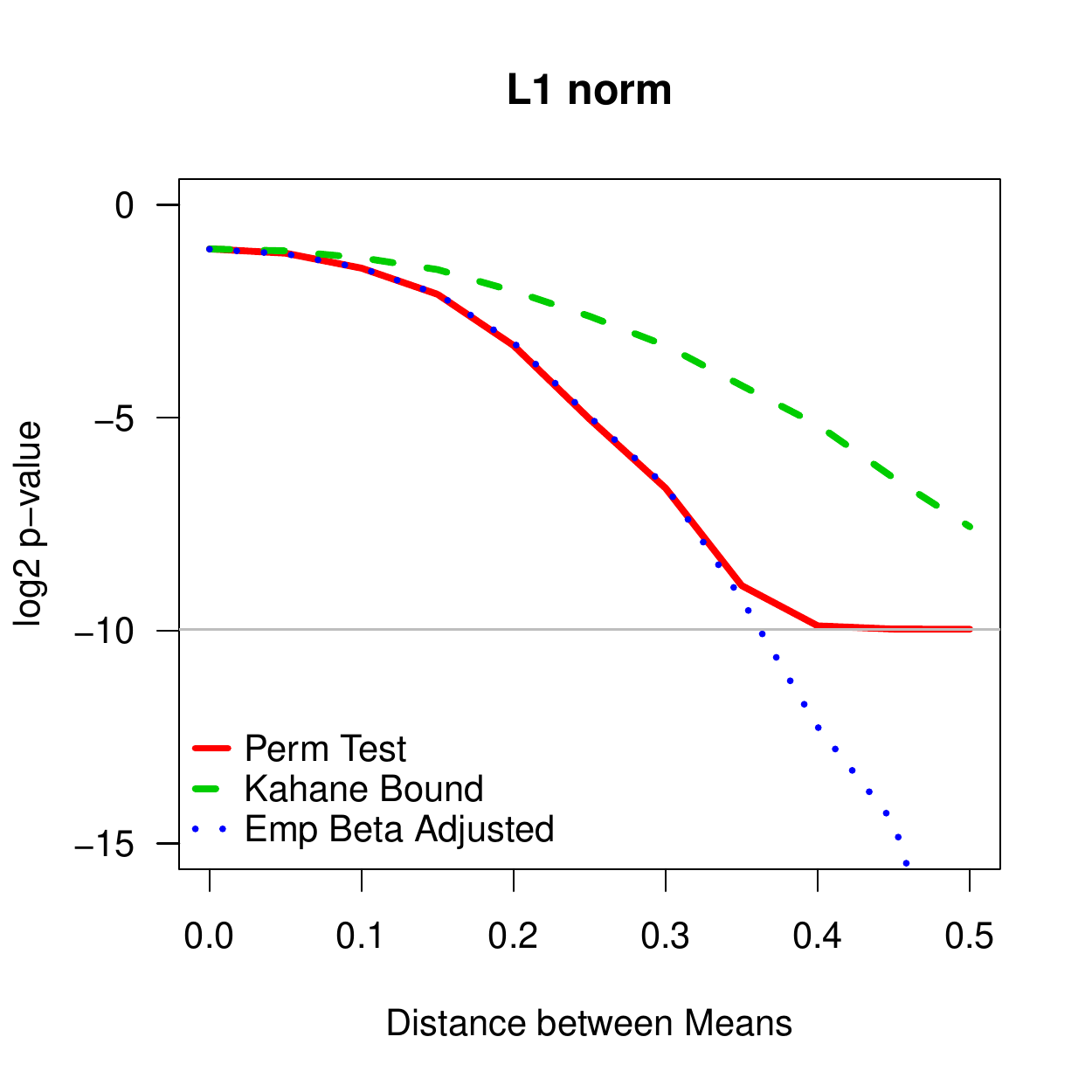}
		\includegraphics[width=0.4\textwidth]{\PICDIR/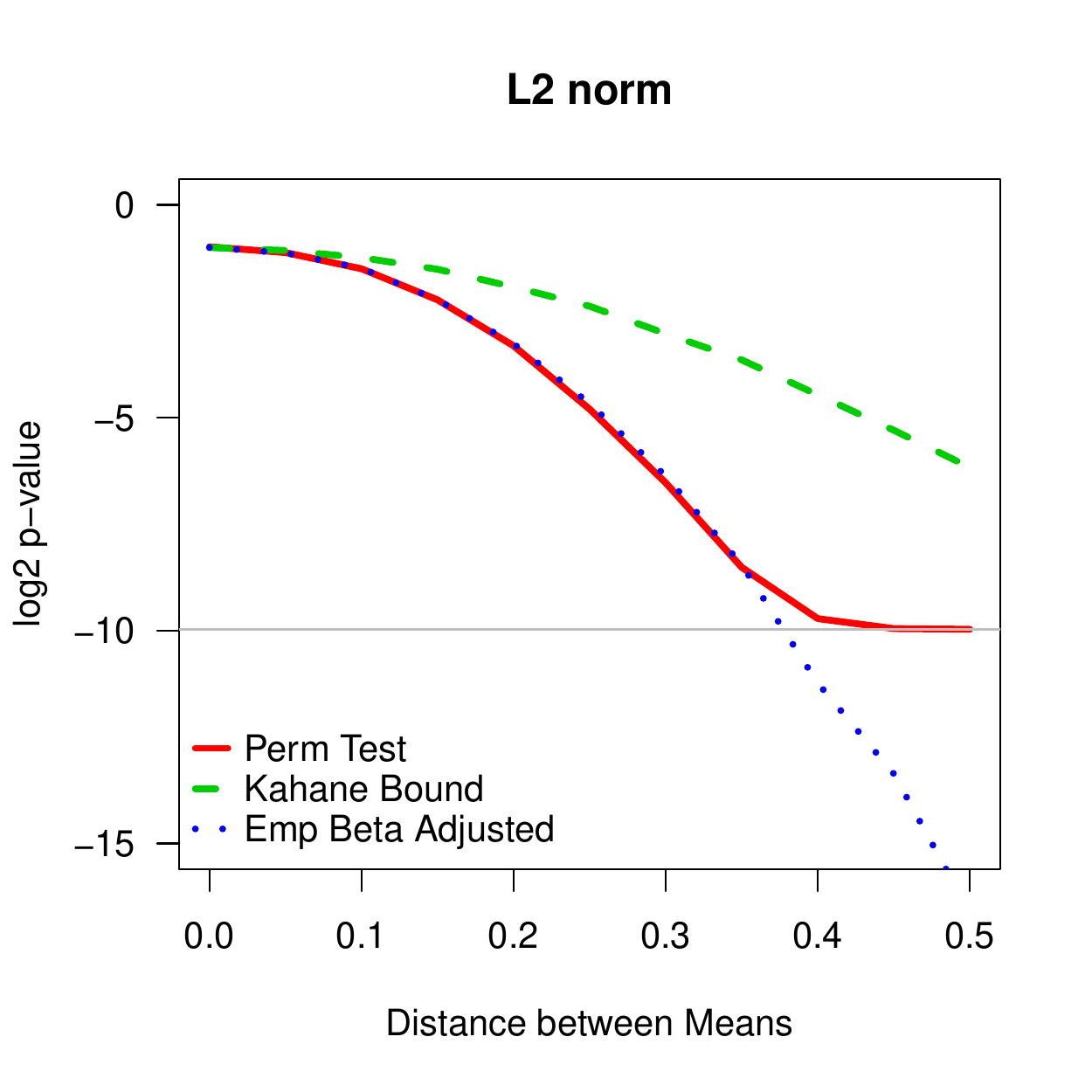}
		\includegraphics[width=0.4\textwidth]{\PICDIR/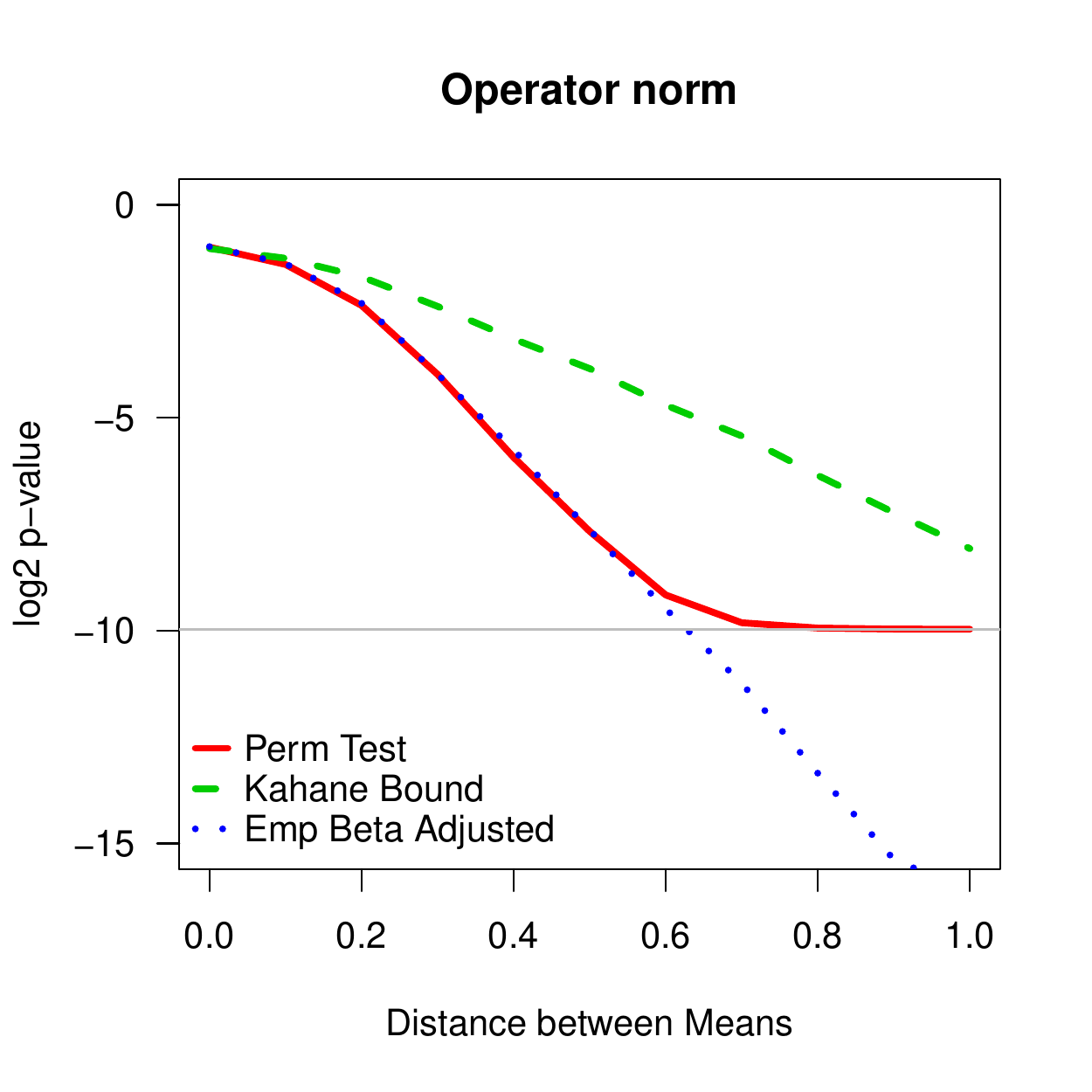}
	\end{center}
	\caption{
		\label{fig:multVTest}
		Multivariate two sample test for normal data with 
		balanced sample sizes $m_1=m_2=50$ 
		for $\ell^1$, $\ell^2$, and $\ell^\infty$ norms.
		The plots compare the permutation test (red) with 
		1000 permutations to
		the Kahane bound from Theorem~2.2 (green) 
		and the beta adjusted Kahane bound (blue) 
		across 1000 replications.  
	}
\end{figure}

\subsection{Berkeley Growth Curves Null Setting}
\label{app:berkNull}

In this section, we repeat the data analysis from
Section~4.2. 
However, we 
first remove the sex labels from the Berkeley growth curve
dataset.  Hence, when sampling two sets of size 30, each 
resample will contain both male and female curves.  Thus,
there should be on average no statistical difference
between the two sets.  Over 100 replications for each
of the three norms $L^1$, $L^2$, and $L^\infty$ as 
well as the two bounds---unadjusted Kahane and 
beta adjusted---we have Figure~\ref{fig:berkNull},
which plots the empirical p-values against the 
theoretical p-values from the uniform distribution
on $[0,1]$.  We see large deviations for the unadjusted
Kahane bound in the $L^2$ and $L^\infty$ norms 
yielding an overly conservative hypothesis test.
Table~\ref{tab:berkNull} displays the results of 
goodness-of-fit testing for the six sets of null
p-values with a similar conclusion.

\begin{table}
	\begin{center}
		\fbox{
			\begin{tabular}{l|rrr||rrr}
				& \multicolumn{3}{c}{Kahane Bound} 
				& \multicolumn{3}{c}{Beta Adjusted Bound}\\\hline
				& $L^1$ & $L^2$ & $L^\infty$ 
				& $L^1$ & $L^2$ & $L^\infty$ \\
				KS test & 3.6\% & $<$0.001\% & $<$0.001\%
				& 8.7\% & 87.9\%     & 9.2\% \\
				AD test & 3.8\% & $<$0.001\% & $<$0.001\%
				& 2.2\% & 79.7\%     & 5.1\% \\
			\end{tabular}
		}
	\end{center}
	\capt{
		\label{tab:berkNull}
		This table contains p-values from both the 
		Kolmogorov-Smirnov and the Anderson-Darling 
		goodness-of-fit tests for the 100 computed
		p-values against the Uniform$[0,1]$ distribution.
	}
\end{table}

\begin{figure}
	\includegraphics[width=\textwidth]{\PICDIR/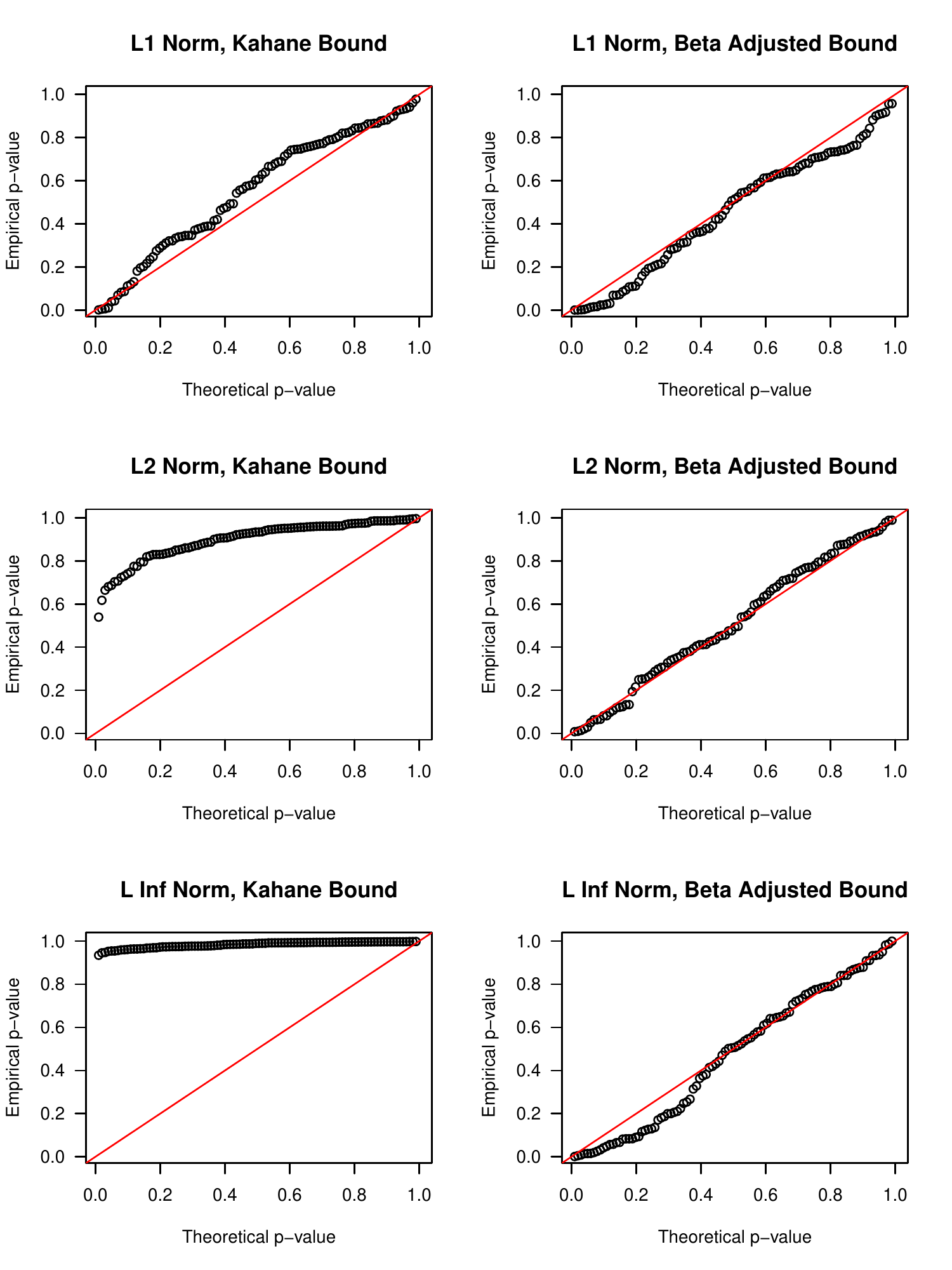}
	\capt{
		\label{fig:berkNull}
		100 simulated null p-values are plotted against the 
		theoretical values from the uniform distribution on 
		$[0,1]$.  There are massive deviations from uniformity
		for unadjusted Kahane bound with the $L^2$ and $L^\infty$ 
		norms.
	}
\end{figure}

\subsection{Phoneme Curves Null Setting}
\label{app:phonemeNull}

Similar to Appendix~\ref{app:berkNull}, we aim to test 
for whether or not the empirical beta adjusted p-values
for the phoneme curves from Section~4.3 
follow a Uniform[0,1] distribution in the null setting.
To do this, we repeat the test from Section~4.3 
but remove the label information.  Hence, the two samples of 
size 40 will comprise operators from both phonemes, and 
there should be no significant difference between the
two samples.

Table~\ref{tab:phmNull} reports p-values from the 
Kolmogorov-Smirnov and Anderson-Darling goodness-of-fit 
tests comparing the empirical distribution of the 100 
two sample test adjusted p-values to a uniform distribution
on the unit interval.  Most of these empirical tests 
yield insignificant p-values especially after taking 
multiple testing into account indicating no noticeable
deviation from uniformity.  Hence, the empirical beta
adjustment is able to account for the overly conservative
nature of the unadjusted Kahane bounds.

\begin{table}
	\begin{center}
		\fbox{
			\begin{tabular}{crrrrcrrrrcrrrr}
				& \multicolumn{14}{c}{\bf Kolmogorov-Smirnov}\\
				& \multicolumn{4}{c}{Trace Norm} &&\multicolumn{4}{c}{Hilbert-Schmidt Norm} &&
				\multicolumn{4}{c}{Operator Norm} \\
				& \textipa{A} &  \textipa{O} & \textipa{d} & \textipa{S} &\phantom{X}
				& \textipa{A} &  \textipa{O} & \textipa{d} & \textipa{S} &\phantom{X}
				& \textipa{A} &  \textipa{O} & \textipa{d} & \textipa{S} \\
				\textipa{O} &10.5&   &   &   && 84.8&   &   &   &&  10.7&   &   &      \\
				\textipa{d} & 0.5& 43.5&   &   && 68.8& 30.4&   &   && 29.4& 42.2&   &  \\
				\textipa{S} &70.5& 25.5& 31.0& && 32.7& 77.3& 30.3&   && 17.5&47.8& 55.2&\\
				\textipa{i} &16.1& 60.3& 41.7& 71.1&&81.0& 77.4& 0.3& 9.1&&
				58.3& 6.2& 86.2& 0.6\\\hline\hline
				& \multicolumn{14}{c}{\bf Anderson-Darling}\\
				& \multicolumn{4}{c}{Trace Norm} &&\multicolumn{4}{c}{Hilbert-Schmidt Norm} &&
				\multicolumn{4}{c}{Operator Norm} \\
				& \textipa{A} &  \textipa{O} & \textipa{d} & \textipa{S} &\phantom{X}
				& \textipa{A} &  \textipa{O} & \textipa{d} & \textipa{S} &\phantom{X}
				& \textipa{A} &  \textipa{O} & \textipa{d} & \textipa{S} \\
				\textipa{O} &16.1&   &   &   && 60.5&   &   &   &&  1.7&   &   &      \\
				\textipa{d} & 0.3& 8.0&   &   && 12.0& 2.4&   &   && 3.1& 2.2&   &  \\
				\textipa{S} &52.2& 19.0& 8.7& && 7.8& 2.4& 14.0&   && 4.3&17.3& 18.0&\\
				\textipa{i} &5.6& 14.8& 11.8& 10.4&&22.2& 15.6& 0.8& 0.05&&
				65.1& 6.5& 65.7& 1.0
		\end{tabular}}
	\end{center}
	\capt{
		\label{tab:phmNull}
		A list of p-values from the Kolmogorov-Smirnov and the 
		Anderson-Darling tests for the two sample tests comparing two
		different phonemes with a sample size of 
		$m_1=m_2=10$ under the trace, Hilbert-Schmidt, and operator
		norms.
	}
\end{table}

\subsection{Simulated Covariance Operator Data}
\label{app:pigoliSim}

In this section, we recreate the two-sample simulation study 
performed in \cite{PIGOLI2014} Section~3 to test our methodology.
Let $\Sigma_M$ and $\Sigma_F$ be the empirical covariance operators
for the male and female Berkeley growth curves, respectively.   
For $\gamma\in[0,6]$, we generate two sets of $n=30$ curves
from a Gaussian process
with mean zero and with covariance operator $\Sigma_M$ for the
first group and 
$
\Sigma(\gamma) = 
[\Sigma_M^{1/2} + \gamma\Delta][\Sigma_M^{1/2} + \gamma\Delta]^\star
$
where $\Delta=\Sigma_F^{1/2}R - \Sigma_M^{1/2}$ and 
$R$ is the operator that minimizes the Procrustes distance
between $\Sigma_M$ and $\Sigma_F$.  Specifically, 
$R = UV^\star$ where $U$ and $V$ come from the singular
value decomposition of 
$(\Sigma_F^{1/2})^\star(\Sigma_M^{1/2}) = UDV^\star$.

For each $\gamma$, we test 
$
H_0: \Sigma_M=\Sigma(\gamma)
$
against
$
H_1:\Sigma_M\ne\Sigma(\gamma)
$
via a standard permutation test as in \cite{PIGOLI2014} with $512$ permutations
and via our Kahane-Khintchine bound.  This is replicated 1000 times
resulting in Figure~\ref{fig:pigoli}.  We see that for the trace, 
Hilbert-Schimdt, and operator norms, the power loss for using our 
upper bound is not much different from the standard permutation test.
At worst, the p-values are 2-4 times larger than necessary.

\begin{figure}
	\includegraphics[width=0.32\textwidth]{\PICDIR/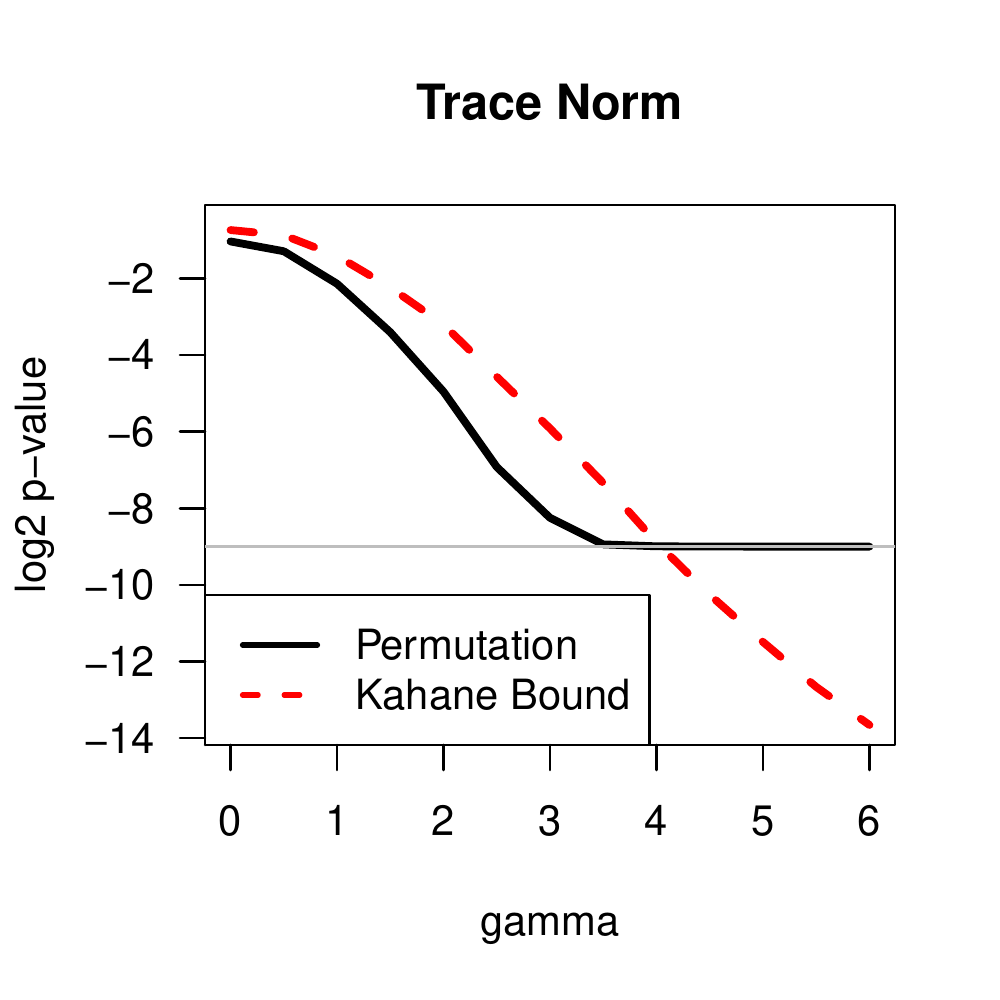}
	\includegraphics[width=0.32\textwidth]{\PICDIR/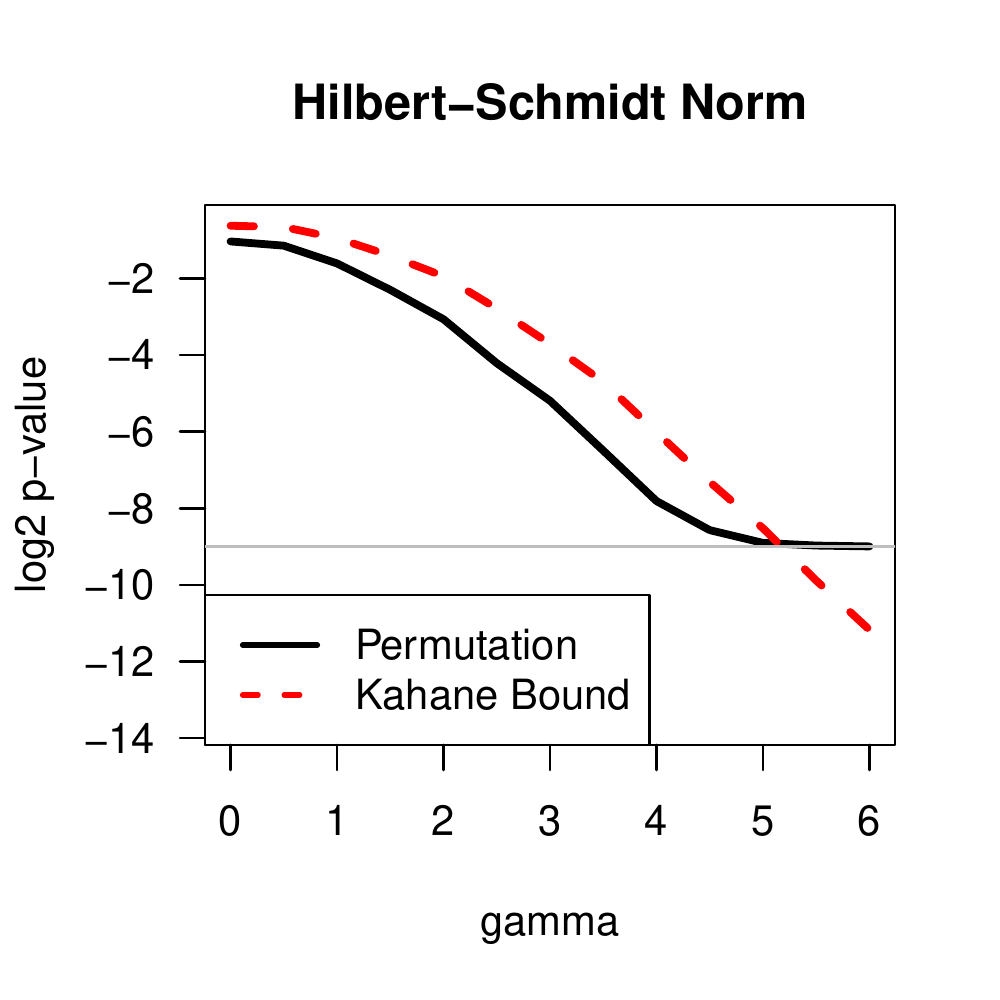}
	\includegraphics[width=0.32\textwidth]{\PICDIR/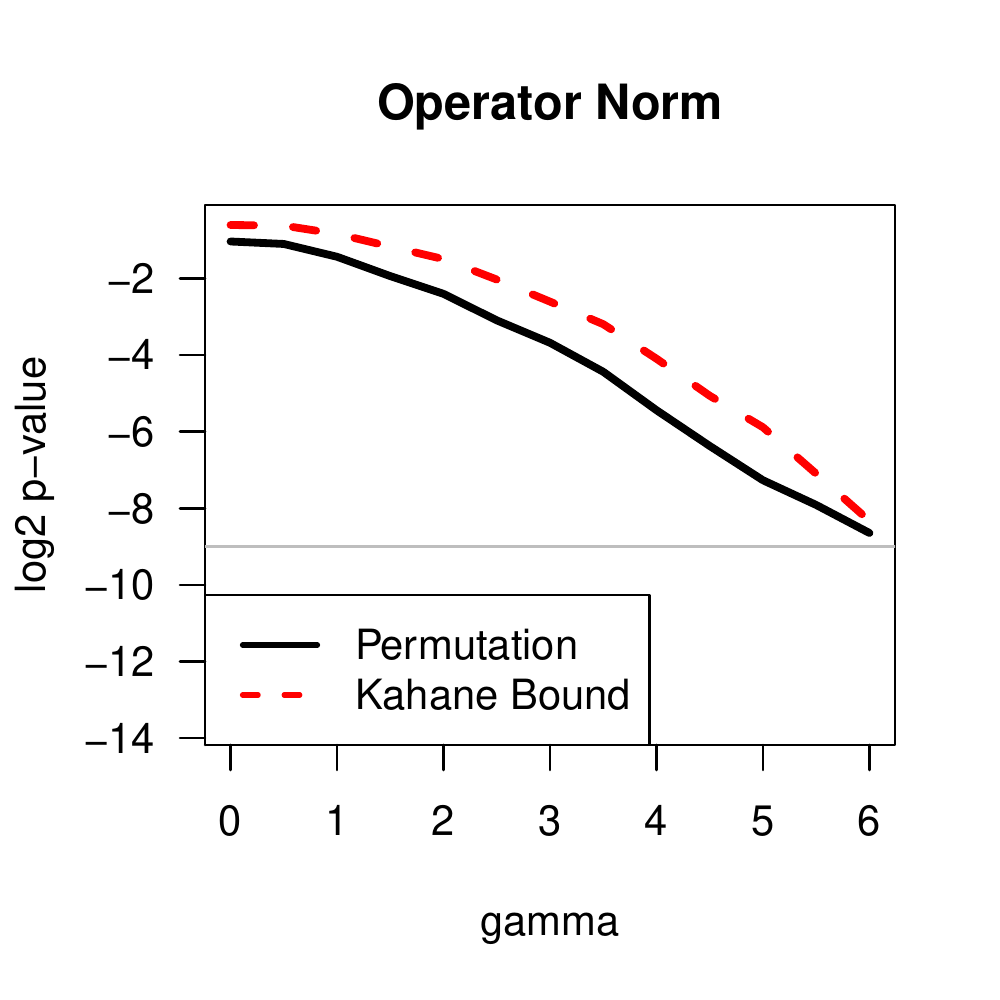}
	\capt{
		\label{fig:pigoli}
		Plotted p-values for a two sample test for equality of 
		covariance operators coming from \cite{PIGOLI2014}.  
		From left to right, the trace, Hilbert-Schmidt, 
		and Operator norms are 
		considered in the three plots.
	}
\end{figure}

\section{Vowel Data}
\label{app:vowelData}

\subsection{Other Schatten Norms}

In this section, we repeat the analysis performed 
in Section~5 
by replacing the trace
norm with both the Hilbert-Schmidt and operator norms.
In 
Figures~\ref{fig:vowelCorrplot2} and~\ref{fig:vowelCorrplotI},
we display results analogous to those seen previously for
the trace norm.
Most notably, as we consider larger values of $q$ for
the $q$-Schatten norms, the number of null hypotheses 
that we fail to reject increases indicating less 
statistical power to distinguish vowel phonemes.
This is in line with much past work using Schatten 
norms on functional data \citep{PIGOLI2014,PIGOLI2018}.

\begin{figure}
	\begin{center}
		\includegraphics[width=\textwidth]{\PICDIR/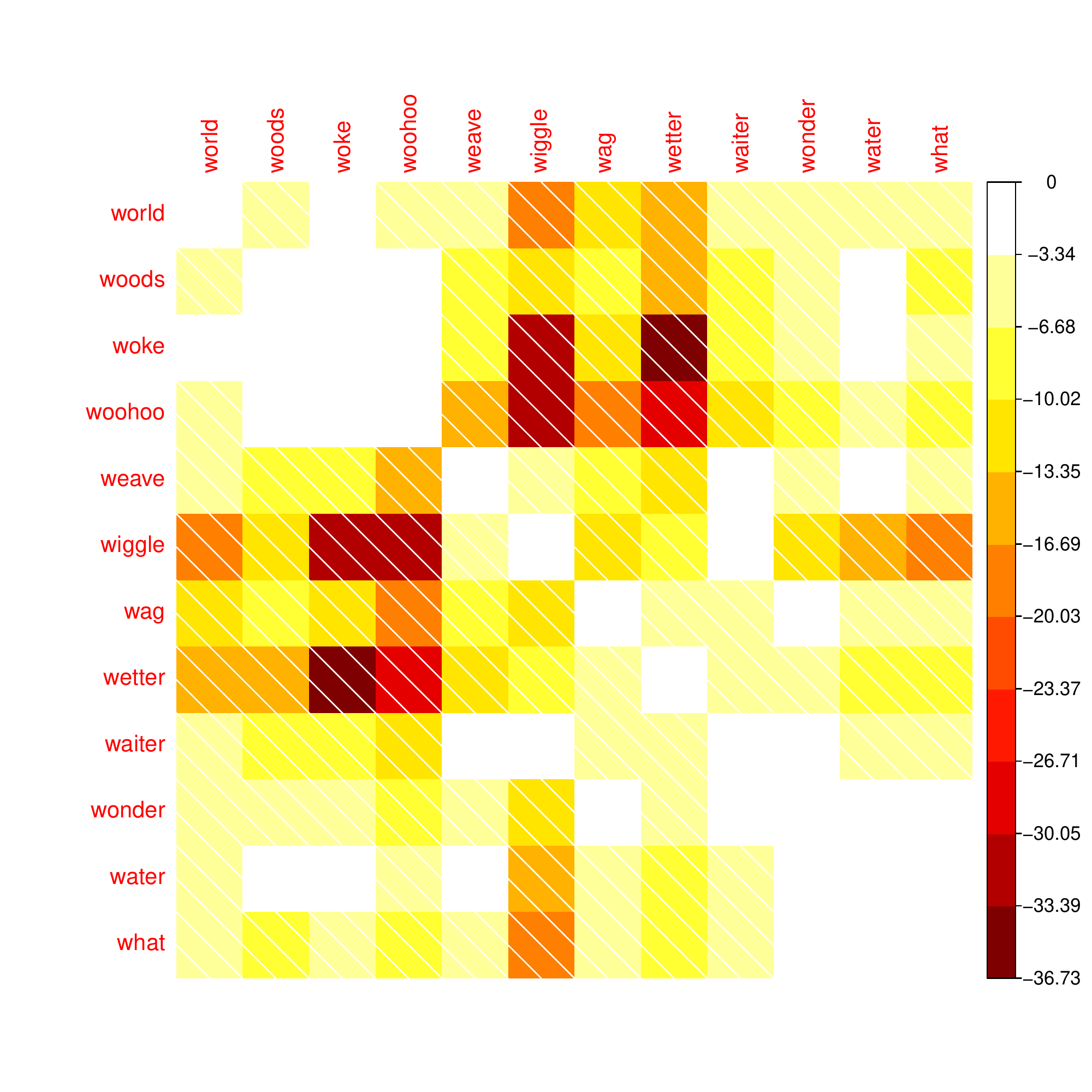}
	\end{center}
	\caption{
		\label{fig:vowelCorrplot2}
		$\log_{10}$ p-values for pairwise two sample tests
		between vowel pairs under the Hilbert-Schmidt norm.
	}
\end{figure}

\begin{figure}
	\begin{center}
		\includegraphics[width=\textwidth]{\PICDIR/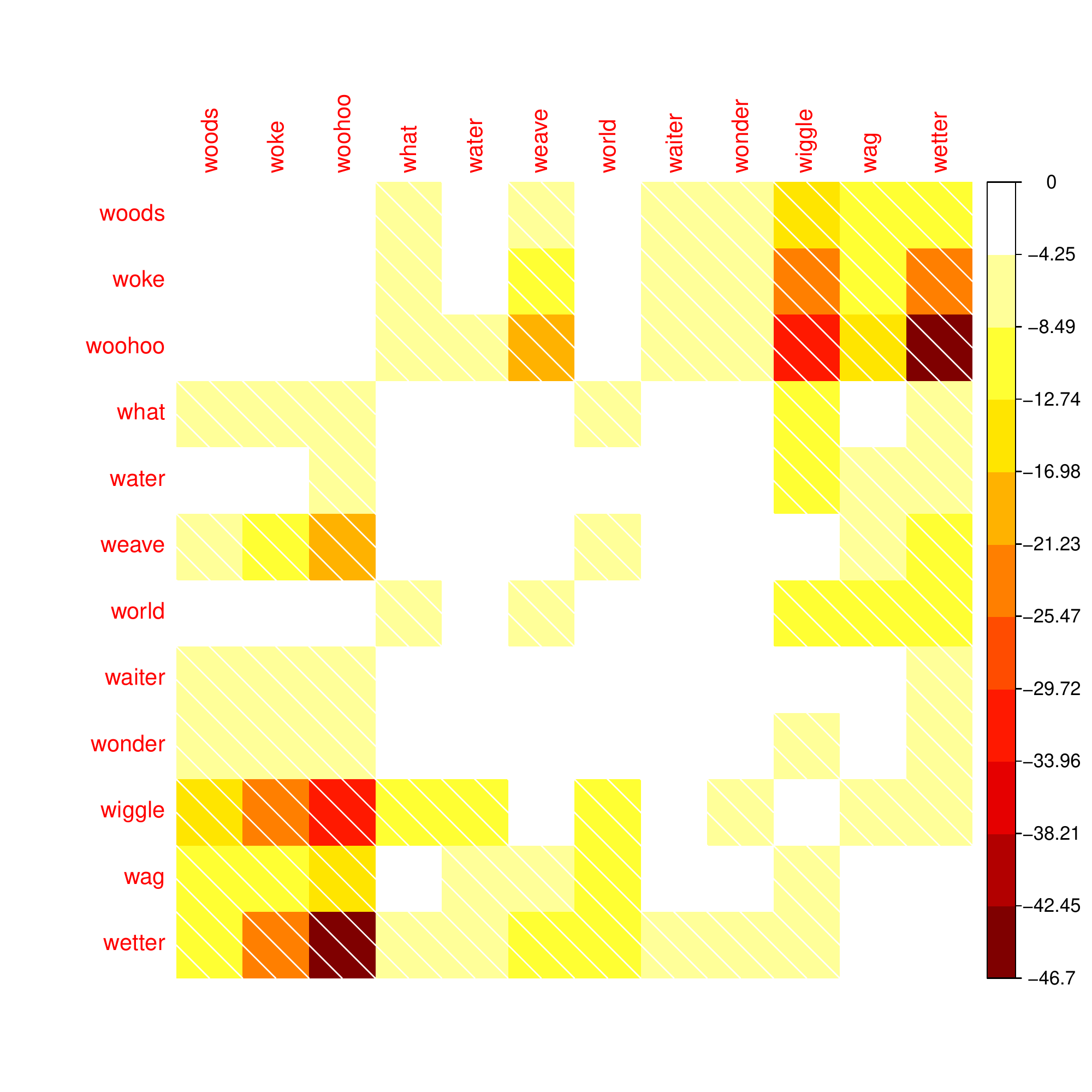}
	\end{center}
	\capt{
		\label{fig:vowelCorrplotI}
		$\log_{10}$ p-values for pairwise two sample tests
		between vowel pairs under the operator norm.
	}
\end{figure}

\subsection{Null Setting}

To check that our methodology, specifically the empirical
beta adjustment from Section~2.4, 
achieves the correct empirical size
and thus is neither conservative nor anti-conservative,
we first randomize all of the labels within each of the Latin 
squares from Section~5. 
Then, we repeat the
same analysis as before.  The 66 p-values produced for each of the
1, 2, and $\infty$ Schatten norms is displayed in Figure~\ref{fig:vowelQQPlot}. These QQ plots compare our empirical p-values to the theoretical
quantiles of the Uniform[0,1] distribution.  For each of the three
norms, we do not see much deviation from uniformity.
Furthermore, for testing goodness-of-fit with the uniform
distribution, the Kolmogorov-Smirnov test returns 
p-values of 0.434, 0.782, and 0.290 and the Anderson-Darling 
test p-values 0.161, 0.511, and 0.241 for the trace, Hilbert-Schmidt,
and operator norms, respectively.  None of these tests are significant
indicating no noticeable deviation from uniformity.

\begin{figure}
	\begin{center}
		\includegraphics[width=0.475\textwidth]{\PICDIR/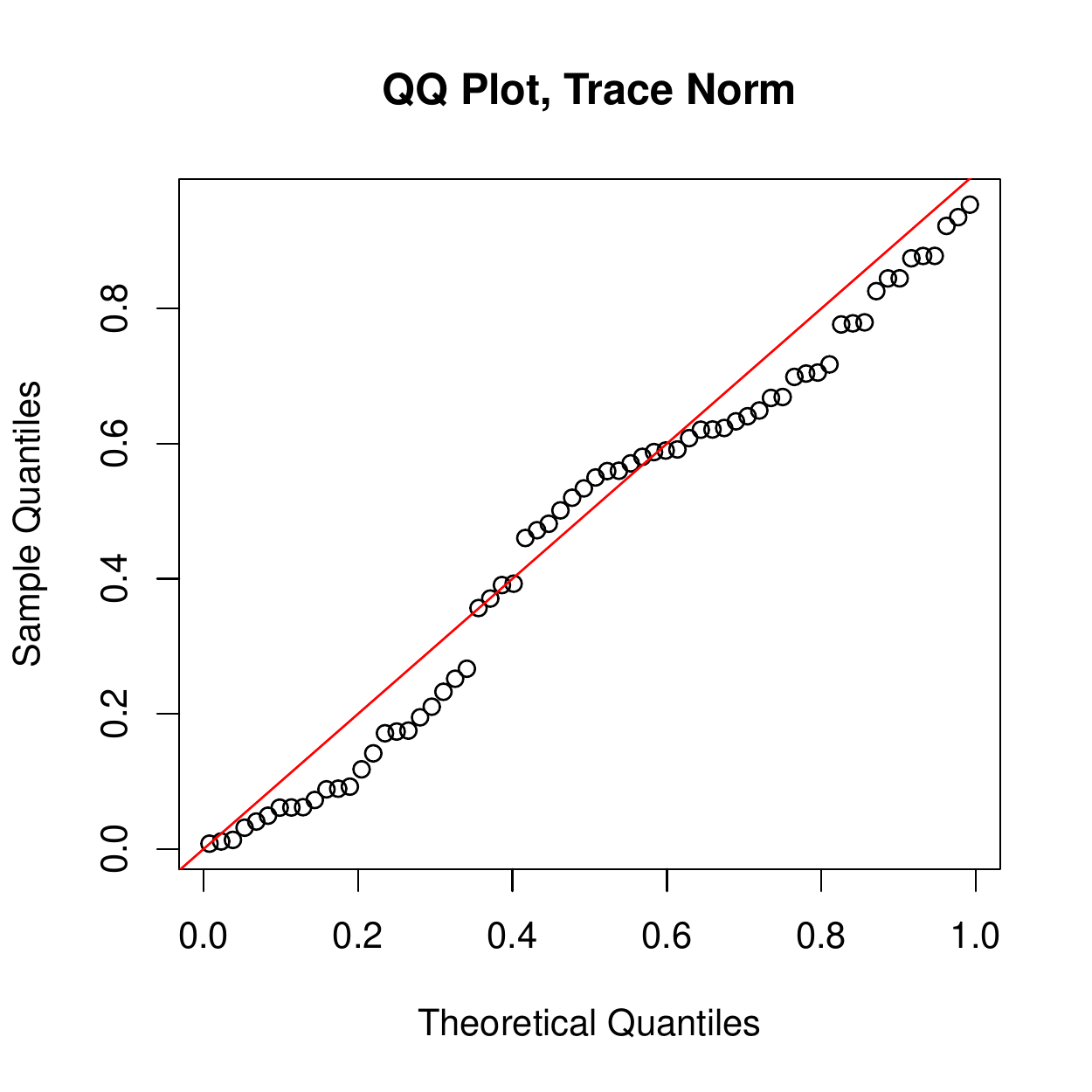}
		\includegraphics[width=0.475\textwidth]{\PICDIR/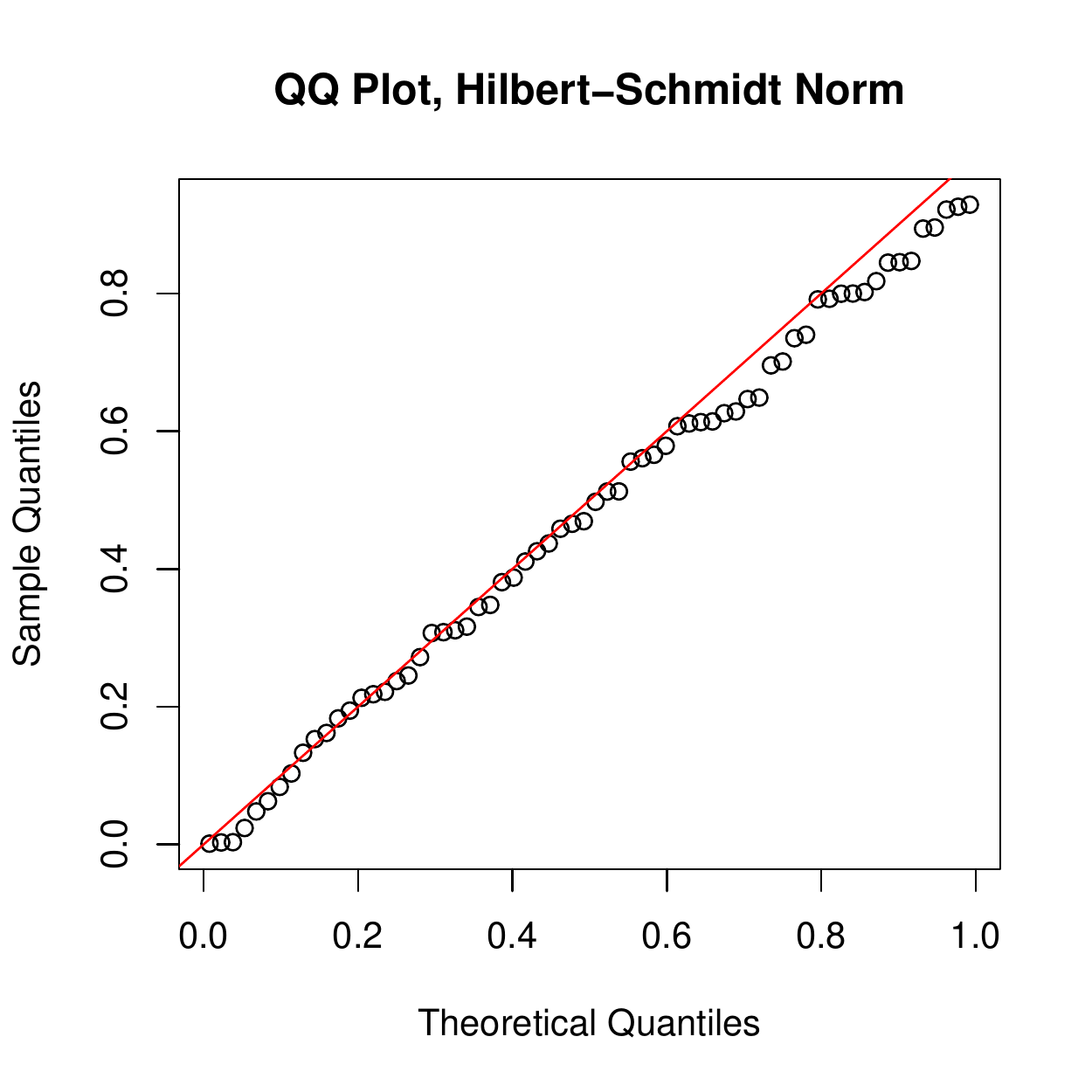}
		\includegraphics[width=0.475\textwidth]{\PICDIR/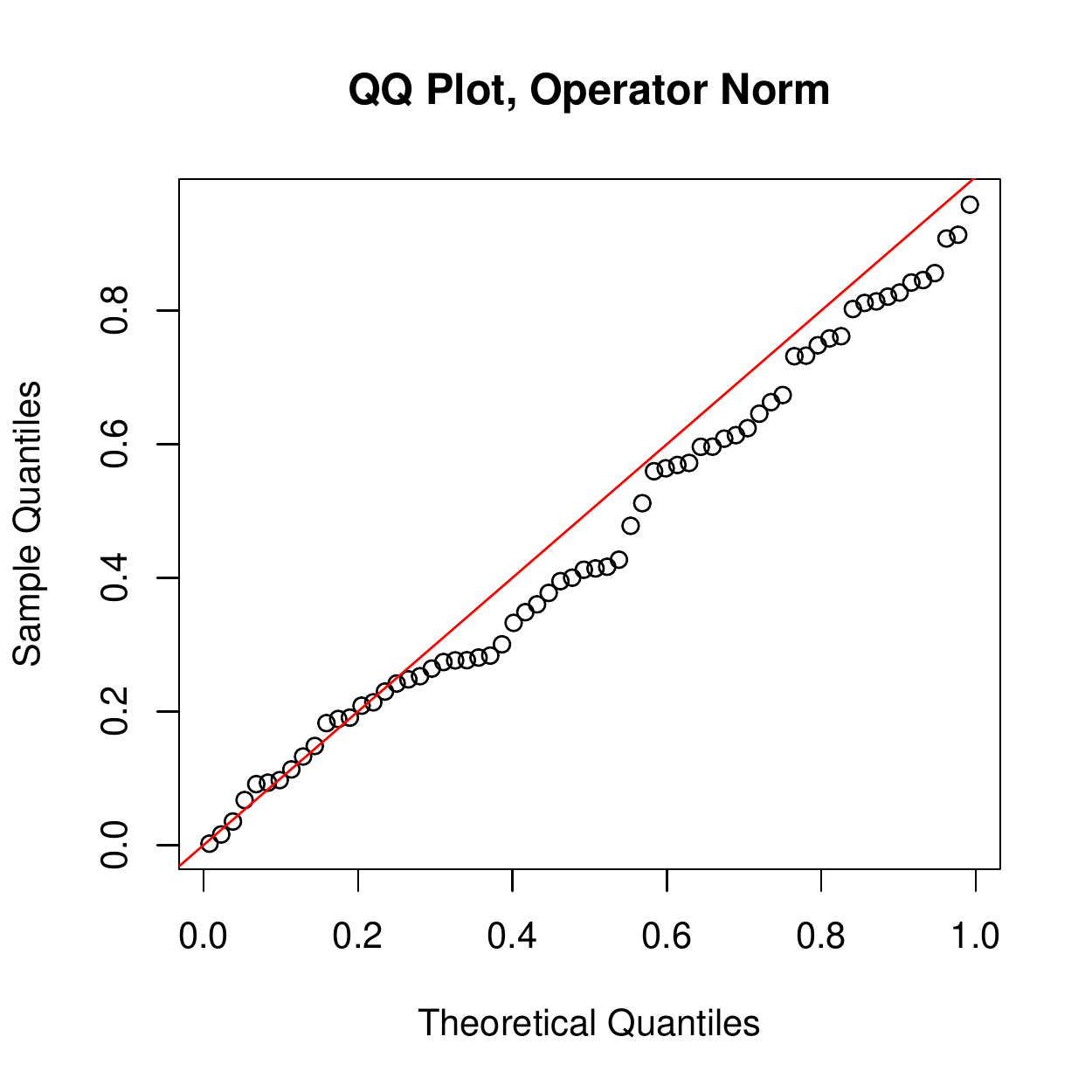}
	\end{center}
	\capt{
		\label{fig:vowelQQPlot}
		QQ Plots comparing the 66 null p-values from the vowel
		data after beta adjustment to the quantiles of the uniform
		distribution.
	}
\end{figure}

\end{document}